\documentclass[12pt]{amsart}

\usepackage[latin1]{inputenc}
\usepackage{amsmath,amsthm,amsfonts,amscd,amssymb,xspace,color,graphicx,mathtools}
\usepackage[english]{babel}
\usepackage{xr}
\usepackage{mathrsfs}

\newtheorem{Prop}{Proposition}[section]
\newtheorem{Lem}[Prop]{Lemma}
\newtheorem{Cor}[Prop]{Corollary}
\newtheorem{Thm}[Prop]{Theorem}

\theoremstyle{remark}
\newtheorem{Rem}[Prop]{Remark}
\newtheorem{Ex}[Prop]{Example}

\numberwithin{equation}{section}

\newcommand{\td}[2]{\xrightarrow[#1\rightarrow #2]{}}
\newcommand{\de}{{\rm d}}

\begin{document}

\title{Directional counting for regular languages}

\author{Rostislav Grigorchuk, Jean-Fran\c cois Quint}

\begin{abstract}
We explain how certain tools from convex analysis and probability theory may be used in order to obtain counting results for the number of words with prescribed frequencies of letters in regular languages.
\end{abstract}

\maketitle

\section{Introduction}

\subsection{General objectives}

The goal of this article (that shares features of survey and research article) is to provide  tools coming from probability theory, dynamical systems and discrete mathematics (automata theory, computer science, theory of formal languages) applicable to the needs of asymptotic combinatorics.

The enumeration problems of combinatorics lead to the coding of combinatorial data by a multivariate power series that could represent a function which would be rational, meromorphic or of a more general type. In the case of one variable, the asymptotic of the coefficients of the corresponding power series could be easily gained from the polynomials representing the function in the rational case or from algebraic equations representing in the algebraic case. There are powerful methods for the asymptotic of coefficients for more general classes of functions (see \cite[Sect. 2.4 \& 2.5]{Melcz}).

In the case of two or more variables, a natural approach is to search for the asymptotic along a certain direction, sometimes called (generalized) diagonal (as is done in \cite{Melcz} for instance). That is, if $F(z)=\sum_{n\in\mathbb N^k} a_n z^n$, $z \in \mathbb C^k$, is a series in $k$ variables, one can fix a probability vector $p=(p_1,\ldots,p_k)$ and inspect asymptotic of coefficients $a_n$ when the index $n=(n_1,\ldots,n_k)$ tends to infinity along the direction given by the vector $p$, that is $n=mp$, $m \in \mathbb N$. This "naive" approach works only if the direction vector $p$ is rational (i.e. has rational coordinates). In this case one considers only those $mp$ that have integer values.

The monograph \cite{Melcz} (based on the work of many mathematicians) provides a bunch of powerful methods and strong results of this sort that serve in the rational case. The irrational case is mentioned time to time in the book but practically no rigorous result is stated.

Here we suggest our approach to the problem. The idea comes from the article \cite{Qui} of the second author and is based on the use of convexity, functional analysis,
and measure theory. To define correctly the directional growth we (using the coefficients of the series) create a Radon measure associated with the (open) cones around the direction vector, then consider the volume growth of this measure and its logarithmic version. This leads us to the definition of the rate of growth in arbitrary direction by means of a homogeneous function, called the indicator of growth. Then we use the condition from \cite{Qui} that we call concave growth  and that guarantees that the involved function is concave and therefore the powerful tools of the convex analysis are applicable.

Then using results of various authors in probability theory, large deviation theory, Perry-Frobenius operators, methods of symbolic dynamics, etc. we provide the results that could serve for a large number of situations. 

The problems of combinatorics that are in the focus of the authors interest have algebraic origin (growth, cogrowth, random walks on groups) and in many situations are transferred to the counting in formal languages over finite alphabet  that could represent a normal form of elements in a group, set of words representing the identity element etc.
In the Chomsky hierarchy of languages the regular languages are at the first level and they are the class of languages mostly used by algebraists. Such languages can be defined by finite automata-acceptors. The corresponding counting problem is closely related to the counting for subshifts of finite type, the first level of complexity in symbolic dynamics. A second level gives sofic systems that also lead to the use of regular languages. 
The material of the last section is adapted to these cases.

This article is a continuation of studies initiated in the previous work \cite{GQS} of authors and A. Shaikh. Making comparison of the results presented here and in monograph \cite{Melcz}, we indicate that in the "nice" examples and for rational directions the function defined here coincides with the logarithm of the function presented in \cite[Definition 6.20]{Melcz}.

At the same  time we have to emphasize that, among other results, we have the Corollary \ref{psiregular} claiming the analytic dependence of the directional growth from the direction.

\subsection{Counting for automata}
We now present the framework of our results.
Given a finite set $A$, we consider a language $W$ written with the alphabet $A$, that is, $W$ is a subset of the set $A^*=\bigcup_{n\geq 0} A^n$ of all finite words over the alphabet $A$. Counting words in $A$ means giving an asymptotic as $n\rightarrow\infty$
for the cardinality of the set $W_n=W\cap A^n$ of words with length $n$ in $W$. 

In what follows, we are interested in a more refined form of counting. Indeed, to each word $w$ in $W$, we can associate a vector whose coordinates describe the number of occurrences of each letter of $A$ in $w$. Formally, we represent this string of integral numbers by an integer valued function $\mathscr P(w)$ on $A$, that is, $\mathscr P(w)$ is an element of $\mathbb N^A$. Now, we raise the following directional counting problem: given a sequence $(f_n)_{n\geq 0}$ of integer valued functions on $A$, may we obtain an asymptotic as $n\rightarrow\infty$ of the number
$N_n=|\{w\in W_n| \mathscr P(w)=f_n\}|$ of words $w$ in $W_n$ with prescribed associated function $f_n$?

For languages associated with finite deterministic automata, it turns out that this is the case, under some natural assumptions on the sequence $(f_n)_{n\geq 0}$. This asymptotic is of the form $N_n\sim n^{-k/2}e^{\delta n}$ where $k$ is an integer which only depends on the automaton and the exponent $\delta>0$ depends on the limit of the direction of the vector $f_n$ in the space $\mathbb R^A$ of all real valued functions on the set $A$.
This is the final result of this article (see Theorem \ref{LLT} and Theorem \ref{LLT2}).
The latter form of the asymptotic will not surprise readers who are familiar with probability theory and in particular with the local limit theorem of Stone \cite{Sto} and the large deviation theory of Bahadur-Rao \cite{BR} and Petrov \cite{Pet}.

Indeed, most of our results on languages generated by automata are essentially known under a different presentation among specialists of the quantitative theory of hyperbolic dynamical systems. They rely on the adaptation to this framework of tools from classical probability theory, following a general principle due to Sinai \cite{Sin} who established a central limit theorem for H\"older continuous observables over such dynamical systems. This theory is presented in the book of Parry and Pollicott \cite{PP}. 

As mentioned above, our final result relies on an adaptation of the proof of the fine large deviation estimates in probability theory. 
Besides the work of Borovkov of Mogul'ski\u{\i} \cite{BoMo} in probability theory, 
we were not able to find a precise reference for such a result over hyperbolic systems. This is surely due to the fact that in the smooth dynamical systems community, people are mainly interested by real valued observables whereas in our case, we will focus on integer valued observables (this splitting appears in probability theory where one separates the study of random walks between the so called non-lattice and lattice cases, see \cite{Rev,Weiss}).

\subsection{Structure of the article}

The article is organized as follows. 

We begin by general ideas about counting. In Section \ref{hahnbanach}, we recall some facts of convex analysis. In Section \ref{expdiv}, we present some tools for studying directional counting in real vector spaces. In Section \ref{concdiv}, under some mild assumptions, we relate those tools with the convex analysis of Section \ref{hahnbanach}. In Section \ref{powerseries}, we draw a link between the previously introduced objects and the approach to counting through higher complex analysis used in the book by Melczer \cite{Melcz}.

Next, we study languages generated by automata. 
In Section \ref{secgraphs}, we introduce general notions for languages; we also define directed graphs and the associated language of paths whose alphabet is the set of edges of the graph. 
We relate the general objects associated to the counting problem for these languages in Section \ref{expdiv} with the spectral properties of certain transfer operator. These operators are a family of Perron-Frobenius operators acting on the space $\mathcal V$ of functions on vertices of the graph. They depend on the choice of a function on the set of edges of the graph. This analysis by means of transfer operators is a direct translation of the theory of \cite{PP} for subshifts of finite type. In Section \ref{secglobalcounting}, we use the previously introduced tools for giving global counting estimates for the language associated with a connected directed graph.

In the remainder of the article we establish our final fine counting results. In the technical Section \ref{seclambda}, we study a certain analytic function $\lambda$ on the space $\mathcal E$ of function on edges of the graph. Given $\theta$ in $\mathcal E$, $\lambda(\theta)$ is the leading eigenvalue of the associated transfer operator. It is a well-known fact in the study of hyperbolic dynamical systems that $\lambda(\theta)$ plays the role of the characteristic function of the law of a random walk in probability theory. Thus, the fine understanding of $\lambda(\theta)$ leads to an analogue of the Central Limit Theorem in Section \ref{seclambda}. Although this result does not come with a very comfortable interpretation in the framework of our original counting problem, the tools which are used in its proof will play a key role in Section \ref{secLLT}, where they are further developed in order to state and prove Theorem \ref{LLT} about directional counting estimates for languages generated by directed graphs. Unfortunately, the statement of the latter result is rather technical. As above, for a given sequence $(f_n)_{n\geq 0}$ of integer valued functions on the set of vertices of the graph, this result gives an asymptotic estimate of the number 
$N_n$ defined above. For this to hold, the sequence $(f_n)_{n\geq 0}$ is subject to a certain set of natural assumptions. For example, the sum of the values of the function $f_n$ must be equal to $n$, otherwise the number $N_n$ is $0$. A less obvious restriction, comes from the fact that, as soon as the graph has at least two vertices, the vectors $\mathscr P(w)$ remain at a bounded distance from a proper subspace of the space $\mathcal E$. Moreover, their projections on that subspace remain close to a proper convex cone which may be smaller than the cone of positive functions. All these objects are introduced precisely along the article and play a role in the final formulation of Theorem \ref{LLT}. In Section \ref{secLLT2}, we establish a version of the directional counting result for languages generated by finite automata: these languages are obtained from languages generated by directed graphs through forgetting part of the information.

\subsection{Acknowledgements} The authors thank the university of Ge\-ne\-va for offering them several occasions to meet at Les Diablerets. The first author was also supported by the Deutsche Forschungsgemeinschaft (DFG, German
Research Foundation) -- SFB-TRR 358/1 2023--491392403,  by  Humboldt  Foundation  and expresses his acknowledgement
to the University of Bielefeld. The authors also thank Nata\v sa Jonoska for helpful discussions on the relation between languages and subshifts.

\section{Convex subsets and homogeneous functions}
\label{hahnbanach}

Let $E$ be a finite-dimensional real vector space. When necessary, we will equip $E$ with a norm but our statements will not depend on the choice of this norm.

Let $\psi:E\rightarrow\mathbb R\cup\{-\infty\}$ be a positively homogeneous function, that is, for $x$ in $E$ and $t$ in $[0,\infty)$, we have $\psi(tx)=t\psi(x)$ (with the convention that $\psi(0)=0$).
If $\psi$ is concave and upper semicontinuous, we associate to $\psi$ the following non empty closed convex subset of the dual space $E^*$ of $E$:
\begin{equation}\label{defomega}\Omega=\{\varphi\in E^*|\forall x\in E\quad\varphi(x)\geq\psi(x)\}.\end{equation}
Conversely, if $\Omega$ is a closed non empty convex subset of $E^*$, we define a homogeneous concave upper semicontinuous function
$$\psi:E\rightarrow\mathbb R\cup\{-\infty\}$$ by setting, for $x$ in $E$,
\begin{equation}\label{relpsiomega}\psi(x)=\inf\{\varphi(x)|\varphi\in\Omega\}.\end{equation}
It follows from the geometric form of Hahn-Banach theorem (see \cite{Rudin}) that those two correspondances are reciprocal to each other.

If $\mathcal C\subset E$ is a closed convex cone, we define its {\em dual cone} $\mathcal C^*$ by
$$\mathcal C^*=\{\varphi\in E^*|\forall x\in E\quad \varphi(x)\geq 0\}.$$
It is a closed convex cone in $E^*$.
Again, the Hahn-Banach theorem implies that the dual cone of $\mathcal C^*$ is $\mathcal C$ (when $E$ is identified with the dual space of $E^*$).

\begin{Lem} \label{interiorcone} Let $\mathcal C\subset E$ be a closed convex cone. Then, an element $x$ in $E$ belongs to the interior $\overset{\circ}{\mathcal C}$ of $\mathcal C$ if and only if, for every $\varphi$ in $\mathcal C^*\smallsetminus\{0\}$, one has $\langle\varphi,x\rangle>0$.
\end{Lem}

\begin{proof} If $\langle \varphi ,x\rangle \leq 0$ for some $\varphi\neq 0$ in $\mathcal C^*$, we choose $y$ in $E$ with $\langle \varphi,y\rangle <0$. Then, for every $\varepsilon >0$, $\langle \varphi ,x+\varepsilon y\rangle< 0$, hence $x$ is not in $\overset{\circ}{\mathcal C}$.

If $\langle \varphi ,x\rangle > 0$ for any $\varphi\neq 0$ in $\mathcal C^*$, then the function
$\varphi\mapsto \langle \varphi ,x\rangle $ is everywhere $>0$ on the compact set $\mathcal C^*\cap S$ where $S$ is the unit sphere of a fixed norm. Therefore, it is bouded away from $0$, which is to say that there exists $\alpha>0$ such that, for any $\varphi$ in $\mathcal C^*$, one has $\langle \varphi ,x\rangle \geq\alpha\|\varphi\|$. In particular, for every $y$ in $E$ with $\|y\|\leq\alpha$, one has,
for $\varphi\in\mathcal C^*$,
$$\langle \varphi ,x+y\rangle \geq \alpha\|\varphi\|-\|y\|\|\varphi\|\geq 0,$$
hence $x+y\in\mathcal C$. Thus, $x$ is an interior points of $\mathcal C$ as required.
\end{proof}

Every concave upper semicontinuous function $\psi$ on $E$ defines a natural closed convex cone, 
namely the closure of the essential domain of definition of $\psi$: 
$$\mathcal C_\psi=\overline{\{x\in E|\psi(x)>-\infty\}}.$$
Every non empty closed convex subset $\Omega$ of $E^*$ defines an other natural closed convex cone, namely the cone
$$\mathcal C_\Omega=\{\varphi\in E^*|\varphi+\Omega\subset\Omega\}.$$

The following lemma tells us that these constructions are dual to each other.

\begin{Lem} \label{translation}
Let $\psi:E\rightarrow\mathbb R\cup\{-\infty\}$ be a homogeneous concave upper semicontinuous function and $\Omega\subset E^*$ be the associated closed convex subset given by \eqref{defomega}. Then we have
$$\mathcal C_\psi^*=\mathcal C_\Omega.$$
\end{Lem}

\begin{proof} First suppose $\varphi$ belongs to $\mathcal C_\psi^*$. Let $\theta$ be in $\Omega$ and $x$ be in $E$. If $x$ is not in $\mathcal C_\psi$, we have $\psi(x)=-\infty$, hence $\varphi(x)+\theta(x)\geq\psi(x)$. If $x$ is in $\mathcal C_\psi$, we have 
$\varphi(x)\geq 0$, hence $\varphi(x)+\theta(x)\geq \theta(x)\geq\psi(x)$. Thus, $\varphi+\theta$ belongs to $\Omega$ as required.

Conversely let $\varphi$ not belong to $\mathcal C_\psi^*$. 
We will prove by contradiction that $\varphi+\Omega$ is not contained in $\Omega$. 
Indeed, suppose $\varphi+\Omega\subset\Omega$, hence $n\varphi+\Omega\subset\Omega$ for any $n\geq 0$.
Choose $x$ in $E$ with $\psi(x)>-\infty$ and $\varphi(x)<0$. Fix $\theta$ in $\Omega$. Then, for $n$ large, we have $n\varphi(x)+\theta(x)<\psi(x)$, which contradicts the fact that $n\varphi+\theta$ belongs to $\Omega$.
\end{proof} 

Note that we have another characterisation of the cone associated to $\Omega$. Define the {\em asymptotic cone} of $\Omega$ as the set of 
$\varphi$ in $E^*$ such that there exists a sequence $(t_n)_{n\geq 0}$ of non-negative real numbers and a sequence $(\varphi_n)_{n\geq 0}$ of elements of $\Omega$ such that 
\begin{equation}\label{defasympcone}t_n\td{n}{\infty}0\mbox{ and }t_n\varphi_n\td{n}{\infty} \varphi.\end{equation}

\begin{Lem} \label{asymptocone} Let $\Omega\subset E^*$ be a non empty closed convex subset. Then the cone 
$\mathcal C_\Omega$
is also the asymptotic cone of $\Omega$.
\end{Lem}

\begin{proof} 
Clearly, the cone $\mathcal C_\Omega=\{\varphi\in E^*|\varphi+\Omega\subset\Omega\}$ is contained in the asymptotic cone. Conversely, if $\varphi$ is in the asymptotic cone, we choose a sequence $(t_n)_{n\geq 0}$ of non-negative real numbers and a sequence $(\varphi_n)_{n\geq 0}$ of elements of $\Omega$ such that 
$t_n\td{n}{\infty}\infty$ and $t_n\varphi_n\td{n}{\infty} \varphi$.
Then, for $n$ large, we have $t_n\leq 1$. In particular, if $\theta$ is in $\Omega$, we get
$(1-t_n)\theta+t_n\varphi_n\in\Omega$. Letting $n$ go to $\infty$ yields $\theta+\varphi\in\Omega $ as required.
\end{proof}

We also show that, in most cases, $\psi$ is defined by the boundary of $\Omega$. Say that a half space in $E^*$ is a convex subset of $E^*$ of the form
$$\{\varphi\in E^*|\varphi(x)\geq a\},$$
where $x$ is a non zero vector in $E$ and $a$ is a real number.

\begin{Lem} \label{halfspace}
Let $\psi:E\rightarrow\mathbb R\cup\{-\infty\}$ be a homogeneous concave upper semicontinuous function and $\Omega\subset E^*$ be the associated closed convex subset. Then either $\Omega=E^*$ or $\Omega$ is a half-space (and $\mathcal C_\psi$ is a half-line) or $\Omega$ is the convex closure of its boundary. In the last case, for $x$ in $E$, we have
$$\psi(x)=\inf_{\varphi\in\partial\Omega}\varphi(x).$$
\end{Lem} 
                                 
\begin{proof} Consider the closed convex cone $\mathcal C_\psi$. 

If $\mathcal C_\psi$ is $\{0\}$, we have $\mathcal C_\psi^*=E^*$, hence, by Lemma \ref{translation}, $\Omega=E$.

If $\mathcal C_\psi$ is a half-line, then there exists $x\neq 0$ in $E$ and $a$ in $\mathbb R$ such that, for $y$ in $E$,
\begin{align*}\varphi(y)&=-\infty&&\mbox{if }y\notin [0,\infty)x\\
\varphi(y)&=ta&&\mbox{if }y=tx,\quad t\geq 0.\end{align*}
Then, a direct computation shows that $\Omega$ is the half-space 
$$\{\varphi\in E^*|\varphi(x)\geq a\}.$$

In every other case, there exists two non zero and non positively colinear vectors $x$ and $y$ with $\psi(x)>-\infty$ and $\psi(y)>-\infty$. Let us show that this implies that $\Omega$ is the convex closure of its boundary. Indeed, as $x$ and $y$ are not positively colinear, there exists a $\theta$ in $E^*$ with $\theta(x)<0$ and $\theta(y)>0$. Pick $\varphi$ in $\Omega$ and consider the set 
$$I=\{t\in \mathbb R| \varphi+t\theta\in\Omega\}.$$
As $\Omega$ is a closed convex set which contains $\varphi$, $I$ is a closed interval in $\mathbb R$ which contains $0$. We claim that $I$ is compact. Indeed, since $\theta(x)<0$ and $\psi(x)>-\infty$, for large $t$, we have
$$\varphi(x)+t\theta(x)<\psi(x)$$
hence $\varphi+t\theta\notin \Omega$ and $t\notin I$. In the same way, for small $t$, we have
$$\varphi(y)+t\theta(y)<\psi(y),$$
hence $t\notin I$. Thus, $I=[a,b]$ for some $a\leq 0$ and $b\geq 0$. Then, $\varphi+a\theta$ and $\varphi+b\theta$ belong to $\partial\Omega$ and hence $\varphi$ belongs to the convex closure of $\partial\Omega$ as required. In particular, for $x$ in $E$ and $\varphi$ in $\Omega$, we have
$$\varphi(x)\geq \inf\{\theta(x)|\theta\in\partial\Omega\},$$
hence 
$$\psi(x)=\inf\{\theta(x)|\theta\in\partial\Omega\}$$
as required.
\end{proof}

We can describe the interior $\overset{\circ}{\Omega}$ of $\Omega$. Note that, as $\psi$ is homogeneous and concave, for any $x$ in $E$, we have $\psi(x)+\psi(-x)\leq 0$.

\begin{Lem}\label{interior} 
Let $\psi:E\rightarrow\mathbb R\cup\{-\infty\}$ be a homogeneous concave upper semicontinuous function and $\Omega\subset E^*$ be the associated closed convex subset. Then the following are equivalent:\\
{\em (i)} the set $\Omega$ has non empty interior.\\
{\em (ii)} for any $x\neq 0$ in $E$, we have $\psi(x)+\psi(-x)<0$.\\
If this holds, we have 
\begin{equation}\label{computeinterior}
\overset{\circ}{\Omega}=\{\varphi\in E^*|\forall x\in E\smallsetminus\{0\}\quad \varphi(x)>\psi(x)\}.
\end{equation}
\end{Lem}

\begin{proof} We actually prove that the negation of the statements are equivalent.

Suppose $\Omega$ has empty interior. Then, as the convex closures of affine basis of $E^*$ have non empty interior, $\Omega$ is contained in an affine subspace of $E^*$. In other words, we may find $x$ in $E\smallsetminus\{0\}$ and $\theta$ in $E^*$ such that
$$\Omega\subset\{\varphi\in E^*|\varphi(x)=\theta(x)\}.$$
In particular, for $t$ in $\mathbb R$, we have
$$\psi(tx)=\inf_{\varphi\in\Omega}t\varphi(x)=t\theta(x),$$
which yields $\psi(x)+\psi(-x)=0$.

Conversely, suppose there exists $x$ in $E\smallsetminus\{0\}$ with $\psi(-x)=-\psi(x)$. If $\varphi$ is in $\Omega$, we get
$$\varphi(x)\geq \psi(x)\mbox{ and }-\varphi(x)=\varphi(-x)\geq \psi(-x)=-\psi(x).$$
Thus, $\varphi(x)=\psi(x)$ for any $\varphi$ in $\Omega$ and $\Omega$ has empty interior.

Now, let us show \eqref{computeinterior}. As $\psi$ is upper semicontinuous, the set 
$$\{\varphi\in E^*|\forall x\in E\smallsetminus\{0\}\quad \varphi(x)>\psi(x)\}$$
is open, hence it is contained in the interior of $\Omega$. Conversely, if $\varphi$ is in $\Omega$ and $\varphi(x)=\psi(x)$ for some $x\neq 0$ in $E$ with $\psi(x)>-\infty$, we choose $\theta$ in $E^*$ with $\theta(x)<0$. Then, for $\varepsilon>0$, $\varphi+\varepsilon\theta$ is not in $\Omega$ and $\varphi$ does not belong to the interior of $\Omega$.
\end{proof}

As a general principle, when a convex object in $E^*$ is smooth, the dual convex object in $E$ is strictly convex. In the following Lemma, we make this precise by assuming that $\Omega$ is real analytic and strictly convex and deducing that $\psi$ is then real analytic and 
strictly concave.

Recall that, if $V$ is a real vector space, a function $f$ defined on an open subset $U$ of $V$ is said to be real analytic if it may be written as the sum of a power series in the neighbourhood of every element of $U$. Equivalently, $f$ is real analytic if it is the restriction to $U$ of a holomorphic function defined on an open subset $U'\supset U$ of the complexification $V_{\mathbb C}$ of $V$. 

\begin{Lem} \label{local}
Let $\rho:E^*\rightarrow\mathbb R$ be a convex real analytic function. Assume that, for every $\theta$ in $E^*$, the derivative
$\de_\theta\rho$ of $\rho$ at $\theta$ does not vanish and that the second derivative $\de^2_\theta\rho$ is positive definite on the kernel of $\de_\theta\rho$. Set $\Omega=\rho^{-1}((-\infty,0])$ and assume that $\Omega$ is non empty. Then, 
$\Omega$ is a closed convex subset of $E^*$.
Let $\psi$ be the concave homogeneous function on $E$ that is dual to $\Omega$. 
Then the cone $\mathcal C_\psi$ has non empty interior and the function $\psi$ is strictly concave and real analytic
in the interior of $\mathcal C_\psi$.
\end{Lem}

Saying that $\psi$ is strictly concave on $\overset{\circ}{\mathcal C}_\psi$ amounts to saying that, for any $x,y$ in $\overset{\circ}{\mathcal C}_\psi$ with 
$y\notin(0,\infty)x$, we have $\psi(x+y)>\psi(x)+\psi(y)$.

\begin{proof} This is a consequence of the implicit function theorem which we apply to the boundary of $\Omega$. Let us be more precise.
We fix a Euclidean norm on $E$, so that the norm is an analytic function outside of $0$. We assume that the dimension of $E$ is $\geq 2$, else the statement is trivial.

Note that, since $\rho$ is a globally defined convex function, it is not bounded, hence $\Omega$ is not equal to all of $E^*$. Pick $\theta$ in $\partial \Omega$. It follows from the implicit function theorem that there exists a neighbourhood $U$ of $\theta$ in $E^*$ such that $\Omega\cap U$ is contained in the half-space $\{\xi\in E^*|\de_\theta\rho(\xi-\theta)\leq 0\}$. As $\Omega$ is convex, all of $\Omega$ is contained in the half-space. In other words, setting $u(\theta)=-\de_\theta\rho$ (which is a linear functional on $E^*$, hence a vector in $E$), we get, for all $\xi$ in $\Omega$,
\begin{equation}\label{minimum}\langle \xi,u(\theta)\rangle\geq \langle \theta,u(\theta)\rangle,\end{equation}
hence $\psi(u(\theta))=\theta(u(\theta))$. In particular $\psi(u(\theta))>-\infty$. Thus, we have built a map 
$u:\partial\Omega\rightarrow E$ such that, for every $\theta$ in $\partial \Omega$, one has
$u(\theta)\in\mathcal C_\psi$. 

Let still $\theta$ be an element of $\partial \Omega$.
The assumption that $\de_\theta^2\rho$ is positive definite on the kernel of $\de_\theta^2\rho$ ensures that the inequality in 
\eqref{minimum} is strict for any $\xi\neq\theta$ close enough to $\theta$ and belonging to $\Omega$. Still as $\Omega$ is convex, the inequality is strict for any $\xi\neq\theta$ in $\Omega$. We set $v(\theta)=\|u(\theta)\|^{-1}u(\theta)$.  
Then, $\theta$ is the unique element of $\Omega$ with $\psi(v(\theta))=\theta(v(\theta))$, hence
the map $v:\partial\Omega\rightarrow E$ is injective. We will now consider this map as an analytic map from
$\partial\Omega$ towards the unit sphere $S$ of $E$ and show that its range is the intersection of 
$\overset{\circ}{\mathcal C}_\psi$ with $S$.

Still by the positivity assumption, for $\theta$ in $\partial\Omega$, the vector $u(\theta)$ does not belong to the 
space ${\de_\theta}u({\rm T}_\theta\partial \Omega)$, that is, to the range of the differential of $u$ on the tangent space to $\partial \Omega$ at $\theta$ (this is the infinitesimal version of the injectivity property above). This says that the differential of $v$ at 
$\theta$ maps this tangent space onto the tangent space of the sphere. Therefore, as $v$ is injective, it is an analytic diffeomorphism onto its image which is an open subset of $S$. 

We claim that this open subset of the sphere is given by 
$v(\partial \Omega)=\overset{\circ}{\mathcal C}_\psi\cap S$. Indeed, the set $(0,\infty)v(\partial\Omega)$ is open in $E$ and contained in $\mathcal C_\psi$. Conversely, let $x\neq 0$ be in $E\smallsetminus (0,\infty)v(S)$ an let 
us show that $x$ does not belong to the interior of $\mathcal C_\psi$. If $\psi(x)=-\infty$, this follows from the definition of $\mathcal C_\psi$. Thus, we now assume $\psi(x)>-\infty$. Note that, as $E$ has dimension $\geq 2$, the positivity property of the second derivative of $\rho$ implies that the convex set $\Omega$ is not a half-space. Therefore, by Lemma \ref{halfspace}, we have 
$$\psi(x)=\inf_{\theta\in\partial\Omega}\theta(x).$$
We fix a sequence $(\theta_n)_{n\geq 0}$ in $\partial\Omega$ with 
$$\theta_n(x)\td{n}{\infty}\psi(x).$$

We claim that $\|\theta_n\|\td{n}{\infty}\infty$. Indeed, if this was not the case, after extracting a subsequence, we could
assume that $(\theta_n)_{n\geq 0}$ converges to some $\theta$ in $\partial\Omega$, which would then satisfy $\theta(x)=\psi(x)$. 
But, then, as $\psi$ is concave and $\leq \theta$, for all $0\leq t\leq 1$, we would have 
\begin{equation}\label{convexequal}\theta(tx+(1-t)v(\theta))=\psi(tx+(1-t)v(\theta)).\end{equation}
As $(0,\infty)v(\partial\Omega)$ is open in $E$, for small $t$, we have $tx+(1-t)v(\theta)=sv(\xi)$ for some $s>0$ and $\xi$ in $\partial \Omega$. From \eqref{convexequal}, we get $\theta(v(\xi))=\psi(v(\xi))$, hence, from the construction of the map $v$, $\xi=\theta$. Thus, $x$ belongs to $\mathbb Rv(\theta)$. As by assumption, $x$ does not belong to $[0,\infty)v(\theta)$, $x$ belongs to $-(0,\infty)v(\theta)$. But then, from 
\eqref{convexequal}, we get $\psi(v(\theta))+\psi(-v(\theta))=0$, hence, by Lemma \ref{interior}, the set $\Omega$ has empty interior, which contradicts the local study of the function $\rho$. Hence, the sequence $(\theta_n)_{n\geq 0}$ goes to $\infty$ in $E^*$.

After extracting, we can assume 
that $\|\theta_n\|^{-1}\theta_n\td{n}{\infty}\varphi$ for some $\varphi\neq 0$ in $E^*$. This linear functional $\varphi$ belongs to the asymptotic cone of $\Omega$, which by Lemma \ref{translation} and Lemma \ref{asymptocone}, is the dual cone of $\mathcal C_\psi$.
As $\theta_n(x)\td{n}{\infty}\psi(x)>-\infty$, we get $\varphi(x)=0$. Thus, we have built an element $\varphi$ in $\mathcal C_\psi^*$ with $\varphi(x)=0$. By Lemma \ref{interiorcone}, this shows that $x$ does not belong to the interior of $\mathcal C_\psi$ as required.

We now have established that $v$ is an analytic diffeomorphism of $\partial\Omega$ onto $S\cap \overset{\circ}{\mathcal C}_\psi$. By construction, for $x$ in $\overset{\circ}{\mathcal C}_\psi$, we have
$$\psi(x)=\langle v^{-1}(\|x\|^{-1}x),x\rangle,$$
which shows that the function $\psi$ is strictly concave and analytic on $\overset{\circ}{\mathcal C}_\psi$.
\end{proof}

\begin{Ex} \label{freepsi}
Set $E=\mathbb R^k$, $k\geq 1$, and let us identify $E^*$ with $\mathbb R^k$ through the standard pairing, that is, 
$$\langle x,y\rangle=x_1y_1+\cdots+x_ky_k,\quad x,y\in\mathbb R^k.$$ 
Then, we claim that the convex subset
$$\Omega=\{\theta\in\mathbb R^k|e^{-\theta_1}+\cdots+e^{-\theta_k}\leq 1\}$$
is in duality with the concave homogeneous function $\psi$ defined by
\begin{align}\label{psientropy}\psi(x)&=\sum_{\substack{1\leq i\leq k\\ x_i>0}}x_i\log\frac{x_1+\cdots+x_k}{x_i}&&x_1,\ldots,x_k\geq 0\\
\nonumber&=-\infty&&\mbox{else}.
\end{align}
This will follow from applying the tools used in the proof of Proposition \ref{local}.

Indeed, we define the function 
\begin{align*}\rho:E&\rightarrow \mathbb R\\ \theta&\mapsto e^{-\theta_1}+\cdots+e^{-\theta_k}-1,\end{align*}
which is clearly real analytic, so that we have $\Omega=\rho^{-1}((-\infty,0])$.
The second derivative of $\rho$ is everywhere positive definite, hence we are in the setting of 
Proposition \ref{local}. We keep the notation of its proof. A direct computation gives, for $\theta$ in $E$, 
\begin{equation}\label{eqexplicitu}u(\theta)=(e^{-\theta_1},\ldots,e^{-\theta_k}).\end{equation}

Let $\psi:E\rightarrow[-\infty,\infty)$ be the upper semicontinuous concave homogeneous function which is dual to $\Omega$. The proof of Proposition \ref{local} shows that the interior of the cone $\mathcal C_\psi$ is the set $(0,\infty) u(\partial\Omega)$. Using \eqref{eqexplicitu}, we obtain 
$$\overset{\circ}{\mathcal C}_\psi=(0,\infty)\{(y_1^{-1},\ldots,y_k^{-1})|y_1,\ldots,y_k>0,y_1+\cdots+y_k=1\},$$
hence clearly, $\mathcal C_\psi=[0,\infty)^k$.

Finally, still by the proof of Proposition \ref{local}, for $x$ in $\overset{\circ}{\mathcal C}_\psi=(0,\infty)^k$, we have 
$\psi(x)=\langle \theta,x\rangle$, where $\theta$ is the unique element of $\partial\Omega$ with $u(\theta)\in (0,\infty)x$. 
A direct computation using \eqref{eqexplicitu} gives
$$\theta_i=\log\frac{x_1+\cdots+x_k}{x_i},\quad 1\leq i\leq k,$$
hence \eqref{psientropy} holds for $x$ in $(0,\infty)^k$. The case of $x$ in $[0,\infty)^k$ follows as $\psi$ is upper semicontinuous.

Note that the restriction of $\psi$ to the simplex 
$$\{x\in [0,\infty)^k| x_1+\cdots+x_k=1\}$$
is Shannon entropy function. 
\end{Ex}

We will later need the following lemma which tells us how the duality works when the objects are transported through linear maps. 

\begin{Lem} \label{concavetransport}
Let $F$ be an other finite-dimensional real vector space and $T:E\rightarrow F$ be a linear map, with adjoint map $T^*:F^*\rightarrow E^*$. Suppose $\psi:E\rightarrow \mathbb R\cup\{-\infty\}$ is a concave upper semicontinuous function and let $\Omega\subset E^*$ be the associated closed convex subset. Assume that $\psi<0$ on $\ker T\smallsetminus\{0\}$. Then, the set $\Omega'=(T^*)^{-1}\Omega\subset F^*$ is not empty and the associated homogeneous concave upper semicontinuous function
$\psi':F\rightarrow \mathbb R\cup\{-\infty\}$ is given by, for $y$ in $F$,
\begin{equation}\label{eqconcavetransport4}
\psi'(y)= \sup\{\psi(x)|x\in E\quad Tx=y\}.\end{equation}
\end{Lem}

In the above formula, we took the convention that $\sup\emptyset=-\infty$.

\begin{proof} Set $G=\ker T$ and, for $y$ in $F$, define $\psi'(y)$ 
to be as in \eqref{eqconcavetransport4}, where for the moment, we consider $\psi'(y)$ as an element of $[-\infty,\infty]$. We will show that the function $\psi'$ actually takes values in $[-\infty,\infty)$ and is concave and upper semicontinuous. This will imply the result.

First, we show that $\psi'<\infty$ everywhere. As $\psi'$ is $-\infty$ on $F\smallsetminus TE$, it suffices to show that we have $\psi'<\infty$ on $TE$. 
We choose a norm on $E$ and we begin by noticing that, as $\psi$ is upper semicontinuous, there exists $\varepsilon>0$ such that, for every $z$ in $G$, we have $\psi(z)\leq-\varepsilon\|z\|$.
Pick $x$ in $E$; we have 
\begin{equation}\label{eqconcavetransport1}
\limsup_{\substack{z\rightarrow\infty\\ z\in G}}\frac{\psi(x+z)}{\|z\|}\leq-\varepsilon.\end{equation}
Indeed, if $(z_n)_{n\geq 0}$ is a sequence which goes to $\infty$ in $G$, up to extracting a subsequence, we can assume that it takes non zero values and that $\|z_n\|^{-1}z_n\td{n}{\infty}u$ for some $u$ in $G$. Then, we have 
$\|z_n\|^{-1}(x+z_n)\td{n}{\infty}u$ and therefore, as $\psi$ is upper semicontinuous,
$$\limsup_{n\rightarrow\infty}\frac{\psi(x+z_n)}{\|z_n\|}=\limsup_{n\rightarrow\infty}\psi(\|z_n\|^{-1}(x+z_n))\leq\psi(u)\leq-\varepsilon,$$
so that \eqref{eqconcavetransport1} holds. In particular we have 
\begin{equation}\label{eqconcavetransport2}
\psi(x+z)\xrightarrow[\substack{z\rightarrow\infty\\ z\in G}]{}-\infty.\end{equation}
Still as $\psi$ is upper semicontinuous, it is bounded on compact sets and therefore, by 
\eqref{eqconcavetransport2}, we have $\sup_{z\in G}\psi(x+z)<\infty$, that is, $\psi'(Tx)<\infty$ as required.

Concavity of $\psi'$ directy follows from the concavity of $\psi$ and we will now show that $\psi'$ is upper semicontinuous.
As above, it suffices to prove that $\psi'$ is upper semicontinuous on $TE$. Now, the proof of \eqref{eqconcavetransport2} shows that, if $K$ is a compact subset of $E$, we may find a compact subset $L$ of $G$ such that, for any $x$ in $K$, we have 
\begin{equation}\label{eqconcavetransport3}
\sup_{z\in G}\psi(x+z)=\sup_{z\in L}\psi(x+z)\end{equation}
(the latter supremum being attained by semicontinuity).
In particular, let $(y_n)_{n\geq 0}$ be a sequence in $TE$ which converges to some $y$ in $TE$. We may choose a sequence $(x_n)_{n\geq 0}$ in $E$, which converges to some $x$ in $E$, such that, for any $n\geq 0$, $y_n=Tx_n$. Then, applying \eqref{eqconcavetransport3}, we may assume that there exists a bounded sequence $(z_n)_{n\geq 0}$ in $G$, such that, for $n\geq 0$, $\psi'(y_n)=\psi(x_n+z_n)$. Up to extracting a subsequence, we may also assume that $z_n\td{n}{\infty}z$ for some $z$ in $G$. As $Tz=0$, we have $T(x+z)=Tx=y$ hence,
by upper semicontinuity,  
$$\limsup_{n\rightarrow\infty}\psi'(y_n)=\limsup_{n\rightarrow\infty}\psi(x_n+z_n)\leq \psi(x+z)\leq \psi'(y).$$
Therefore, the function $\psi'$ is upper semicontinuous. 

Now, for $\theta$ in $F^*$ we have 
$$(\theta\geq \psi')\Leftrightarrow (\forall x\in E\quad \psi(x)\leq \theta(Tx))\Leftrightarrow (T^*\theta\in \Omega)\Leftrightarrow (\theta\in \Omega').$$
This shows that $\Omega'$ is not empty and dual to $\psi'$.
\end{proof}

\section{Exponential divergence of measures}
\label{expdiv}

We recall some facts from \cite{Qui}.

Let $\mu$ be a Radon measure on $E$. Later, we will only be concerned with counting measures, but the formalism turns out to be easier to write this way. Fix a norm on $E$. For $x$ in $E$ and $r>0$, we denote by $B(x,r)$ the closed ball with radius $r$ and center at $x$ in $E$.
The following Lemma is independent on the particular choice of the norm.

\begin{Lem} Let $\mu$ be a Radon measure on $E$. The followig are equivalent:\\
{\em (i)} there exists some $s$ in $\mathbb R$ such that 
\begin{equation}\label{defsubexpdiv}\mathcal F(s)=\int_E\exp(-s \|x\|)\de\mu(x)<\infty.\end{equation}
{\em (ii)} there exists $a,C\geq 0$ such that, for any $r>0$,
$$\mu(B(0,r))\leq Ce^{aR}.$$
\end{Lem}

If these conditions are satisfied, we shall say that $\mu$ has {\em subexponential divergence}. 

\begin{proof} {\em (i)}$\Rightarrow${\em (ii)} As the function $s\mapsto \int_E\exp(-s \|x\|)\de\mu(x)$ is non increasing, we may find 
$s\geq 0$ such that \eqref{defsubexpdiv} holds. Then, by Chebyshev inequality, for $r>0$, we have
$$\mu(B(0,r))\leq e^{sr}\int_E\exp(-s \|x\|)\de\mu(x).$$
{\em (ii)}$\Rightarrow${\em (i)} Let $a$ and $C$ be as in the statement and choose $s>a$. By Fubini Theorem, we have
\begin{align*}\int_E\exp(-s \|x\|)\de\mu(x)<\infty
&=\int_0^\infty\mu(\{x\in E|\exp(-s \|x\|)\geq u\})\de u\\
&=\int_0^1\mu(B(0,-s^{-1}\log u))\de u\\
&\leq C\int_0^1u^{-\frac{a}{s}}\de u<\infty,
\end{align*}
where the latter follows from the assumption that $a<s$.
\end{proof}

When $\mu$ has subexponential divergence, we set $\tau\in\mathbb R\cup\{-\infty\}$ to be the infimum of all $s$ such that \eqref{defsubexpdiv} holds and we call $\tau$ the exponent of growth of $\mu$. For $s$ in $\mathbb R$, we have
\begin{equation}\label{convergence1}
s>\tau\Rightarrow \mathcal F(s)<\infty\mbox{ and }\mathcal F(s)<\infty\Rightarrow s\geq\tau.\end{equation}
The exponent $\tau$ depends on the choice of the norm. Below, we will introduce a more intrinsic notion.

First, if $\mathcal C\subset E$ is an open cone, we denote by $\tau_\mathcal C$ the exponent of growth of the restriction of $\mu$ to $\mathcal C$.
We set $\psi_\mu(0)=0$ and, for $x$ in $E$, $x\neq 0$, 
$$\psi_\mu(x)=\|x\|\inf\{\tau_\mathcal C|\mathcal C\mbox{ is an open cone and }x\in\mathcal C\}.$$

The function $\psi_\mu$ does not depend on the choice of the norm and is called the indicator of growth of $\mu$. 
It is upper semicontinuous and we have
$$\tau=\sup_{x\neq 0}\frac{\psi_\mu(x)}{\|x\|}.$$
We also get the following multi-dimensional analogue of \eqref{convergence1}. For $\varphi$ in $E^*$, we have 
\begin{equation}\label{phiomega1}(\forall x\neq 0\quad \varphi(x)>\psi_\mu(x))\Rightarrow \int_E\exp(-\varphi(x))\de\mu(x)<\infty\end{equation}
and
\begin{equation}\label{phiomega2}\int_E\exp(-\varphi(x))\de\mu(x)<\infty\Rightarrow(\forall x\neq 0\quad \varphi(x)\geq\psi_\mu(x)).\end{equation}

For counting measures, the function $\psi_\nu$ takes its values in the set $[0,\infty)\cup\{-\infty\}$ (appearance of the value $-\infty$ is justified in Lemma \ref{psicounting} below).
Recall that the asymptotic cone of a non-empty subset $X$ of $E$ is defined by \eqref{defasympcone}.

\begin{Lem} \label{psicounting}
Let $A$ be a countable set and $w:A\rightarrow E$ be a proper map, that is, $w^{-1}(B)$ is finite for any bounded set $B$ in $E$. We set $\mu=\sum_{a\in A}\delta_{w(a)}$ and we assume that $\mu$ has subexponential divergence as in \eqref{defsubexpdiv}, 
that is, there exists $s$ in $\mathbb R$ with
$$\sum_{a\in A}\exp(-s\|w(a)\|)<\infty.$$
Let $\mathcal C$ be the asymptotic cone of $w(A)$. 
Then, we have 
$\psi_\mu\geq 0$ on $\mathcal C$ and $\psi_\mu=-\infty$ on $E\smallsetminus\mathcal C$.
\end{Lem}

\begin{proof} Let $x\neq 0$ be in $E$.

If $x$ is not in $\mathcal C$, there exists an open cone $\mathcal D$ containing $x$ such that $w^{-1}(\mathcal D)$ is finite. Then, for every $s$ in $\mathbb R$, we have 
$$\sum_{\substack{a\in A\\ w(a)\in\mathcal D}}\exp(-s\|w(a)\|)<\infty,$$
hence $\tau_{\mathcal D}=-\infty$ and $\psi_\mu(x)=-\infty$.

If $x$ belongs to $\mathcal C$, then, for every open cone containing $x$, the set $w^{-1}(\mathcal D)$ is infinite. We get
$$\sum_{\substack{a\in A\\ w(a)\in\mathcal D}}1=|\{a\in A|w(a)\in\mathcal D\}|=\infty,$$
hence $\tau_{\mathcal D}\geq 0$. Thus, $\psi_\mu(x)\geq 0$ as required.
\end{proof}

We shall need to study the behaviour of the indicator of growth when mapping measures through a surjective linear map.

\begin{Prop}\label{projection}
Let $F$ be a subspace of $E$ and $\pi:E\rightarrow E/F$ be the quotient map. Let $\mu$ be a Radon measure with subexponential divergence on $E$ and suppose that $\psi_\mu(x)=-\infty$ for every $x\neq 0$ in $F$. Then, the measure 
$\nu=\pi_*\mu$ is finite on compact sets and has subexponential divergence. For every $y$ in $E/F$, we have
$$\psi_{\nu}(y)=\sup\{\psi_\mu(x)|x\in E\quad \pi(x)=y\}.$$
\end{Prop}

The proof uses the assumption on $F$ under the following form.

\begin{Lem} \label{negligible}
Let $F$ be subspace of $E$ and $\mu$ be a Radon measure with subexponential divergence on $E$ and 
$\psi_\mu(x)=-\infty$ for every $x\neq 0$ in $F$.
Then, for every $a$ in $\mathbb R$, there exists 
an open cone $\mathcal C\subset E$ with $F\smallsetminus \{0\}\subset \mathcal C$ and 
$\tau_{\mathcal C}\leq a$.
\end{Lem}

\begin{proof} By assumption, for every unit vector $x$ in $F$, there exists an open cone $\mathcal C_x$ containing $x$ with 
$\int_{\mathcal C_x}\exp(-a\|z\|)\de\mu(z)<\infty$. By compactness of the unit ball of $F$, 
we can therefore find $x_1,\ldots,x_r$ in $F\smallsetminus\{0\}$ with 
$$F\smallsetminus\{0\}\subset \mathcal C=\mathcal C_{x_1}\cup\ldots\cup\mathcal C_{x_r}.$$
We get
$$\int_{\mathcal C}\exp(-a\|z\|)\de\mu(z)\leq \sum_{i=1}^r \int_{\mathcal C_{x_i}}\exp(-a\|z\|)\de\mu(z)<\infty,$$
hence $\tau_{\mathcal C}\leq a$ as required.
\end{proof}

\begin{proof}[Proof of Proposition \ref{projection}] 
We equip $E/F$ with the quotient norm defined by 
$$\|y\|=\inf\{\|x\||x\in E\quad\pi(x)=y\}.$$

We first show that $\nu$ has subexponential divergence. Indeed, by the assumption and Lemma \ref{negligible}, there exists an open cone $\mathcal C$ containing $F\smallsetminus\{0\}$ with $\mu(\mathcal C)<\infty$. As $\mathcal C$ is open, there exists $\varepsilon>0$ such that, for every $x$ in $E\smallsetminus \mathcal C$, one has $\|\pi(x)\|\geq \varepsilon\|x\|$. Now, we choose $s>0$ with
$\int_{E}\exp(-s\|x\|)\de\mu(x)<\infty$ and we dominate
\begin{multline*}\int_{E/F}\exp(-\varepsilon^{-1}s\|y\|)\de\nu(y)=\int_E\exp(-\varepsilon^{-1}s\|\pi(x)\|)\de\mu(x)
\\ \leq\mu(\mathcal C)+\int_{E\smallsetminus\mathcal C}\exp(-s\|x\|)\de\mu(x)<\infty,\end{multline*}
hence $\nu$ has subexponential growth.

Now, we prove the formula for $\psi_\nu$. Take $y$ in $E/F$. If $y=0$, the formula holds directly from the assumption. 

If $y\neq 0$, we take $x$ in $E$ with $\pi(x)=y$ and we show that $\psi_\mu(x)\leq\psi_\nu(y)$. Set $\alpha=\|x\|\|y\|^{-1}$ and fix $\mathcal D\subset E/F$ to be an open cone containing $y$.
For $\varepsilon>0$, the cone 
$$\mathcal C=\{x'\in\pi^{-1}\mathcal D|(1+\varepsilon)^{-1}\alpha\|\pi(x')\|<\|x'\| <(1+\varepsilon)\alpha\|\pi(x')\| \}$$
is open and contains $x$. 
For $s\geq 0$, we have 
$$\int_{\mathcal C} \exp(-s\|x'\|)\de\mu(x')\leq \int_{\mathcal D} \exp(-s(1+\varepsilon)^{-1}\alpha\|y'\|)\de\nu(y')$$
and, for $s\leq 0$,
$$\int_{\mathcal C} \exp(-s\|x'\|)\de\mu(x')\leq \int_{\mathcal D} \exp(-s(1+\varepsilon)\alpha\|y'\|)\de\nu(y').$$
This yields
$$\tau_{\mathcal C,\mu}\leq \alpha^{-1}\max((1+\varepsilon)\tau_{\mathcal D,\nu},(1+\varepsilon)^{-1}\tau_{\mathcal D,\nu}),$$
hence
$$\psi_\mu(x)\leq\|y\|\max((1+\varepsilon)\tau_{\mathcal D,\nu},(1+\varepsilon)^{-1}\tau_{\mathcal D,\nu}).$$
Letting $\varepsilon$ go to $0$, we get $\psi_\mu(x)\leq \|y\| \tau_{\mathcal D,\nu}$ and therefore, as this holds for any $\mathcal D$, 
$\psi_\mu(x)\leq \psi_\nu(y)$ as required. 

To prove the converse inequality, we start by fixing $a<0$ in $\mathbb R$ with 
$a\leq \|y\|^{-1}\sup_{x\in E_y}\psi_\mu(x)$, where 
$$E_y=\pi^{-1}(y)=\{x\in E|\pi(x)=y\}.$$ Then, we apply Lemma \ref{negligible} so that we get an open cone $\mathcal C_0$ with $F\smallsetminus\{0\}\subset\mathcal C_0$ and $\tau_{\mathcal C_0}\leq a$.

Note that the set $K=E_y\smallsetminus \mathcal C_0$ is compact. Indeed, as $x$ goes to $\infty$ in $E_y$, the vector 
$\|x\|^{-1}\pi(x)=\|x\|^{-1} y$ goes to $0$ in $E/F$, that is, the distance between $\|x\|^{-1} x$ and $F$ goes to $0$. Thus, there exists $M\geq 0$ such that, for any $x$ in $E_y$ with $\|x\|\geq M$, we have $x\in\mathcal C_0$. 

Therefore, for $\varepsilon>0$, we can find $x_1,\ldots,x_r$ in $K$ and open cones $\mathcal C_1,\ldots,\mathcal C_r$ containing
$x_1,\ldots,x_r$ such that
$$K\subset \mathcal C_1\cup\ldots\cup\mathcal C_r,$$
and, for $1\leq i\leq r$, $\|x_i\|\tau_{\mathcal C_i}\leq\psi_\mu(x_i)+\varepsilon$ and
$$\mathcal C_i\subset\{x\in E|(1+\varepsilon)^{-1}\alpha_i\|\pi(x)\|<\|x\| <(1+\varepsilon)\alpha_i\|\pi(x)\|\},$$
with $\alpha_i=\|y\|^{-1}\|x_i\|$.

We set $\mathcal C=\mathcal C_0\cup\mathcal C_1 \ldots\cup\mathcal C_r$, so that $E_y\subset\mathcal C$. We define 
$$\mathcal D=\{x\in E|x+F\subset \mathcal C\}$$ 
and we claim that $\mathcal D$ is open in $E$. 

Indeed, suppose this is not the case. Then, we may find $x$ is in $\mathcal D$ and sequences $(x_n)_{n\geq 0}$ in $E$ 
and $(z_n)_{n\geq 0}$ in $F$ with 
$x_n\td{n}{\infty} x$ but, for any $n\geq 0$, $x_n+z_n\notin \mathcal C$. If the sequence $(z_n)_{n\ge 0}$ would go to $\infty$, the distance between the vectors $\|x_n+z_n\|^{-1}(x_n+z_n)$ and the space $F$ would go to zero and hence, for $n$ large, we would have $x_n+z_n\in\mathcal C_0$, which we have assumed to be false. Hence, we can assume that $(z_n)_{n\ge 0}$ is bounded and even that it converges to some $z$ in $F$. Then by assumption, we have $x+z\in\mathcal C$ and therefore, as $\mathcal C$ is open, for $n$ large, $x_n+z_n$ also belongs to $\mathcal C$, a contradiction.

Therefore, the cone $\mathcal D$ is open. By construction, it is stable by translations by $F$, which amounts to say that 
$\mathcal D=\pi^{-1}\pi(\mathcal D)$. As $\mathcal D$ contains $E_y$, $\pi(\mathcal D)$ contains $y$ and, for $s$ in $\mathbb R$, we have
\begin{multline*}\int_{\pi(\mathcal D)}\exp(-s\|y'\|)\de\nu(y')=\int_{\mathcal D}\exp(-s\|\pi(x)\|)\de\mu(x)\\
\leq \sum_{i=0}^r \int_{\mathcal C_i}\exp(-s\|\pi(x)\|)\de\mu(x).\end{multline*}
For $i=0$, we have, if $s\geq 0$, 
$$\int_{\mathcal C_0}\exp(-s\|\pi(x)\|)\de\mu(x)\leq \mu(\mathcal C_0)<\infty$$
and, if $s\leq 0$,
$$\int_{\mathcal C_0}\exp(-s\|\pi(x)\|)\de\mu(x)\leq \int_{\mathcal C_0}\exp(-s\|x\|)\de\mu(x),$$
which is finite as soon as $s>a$.
In the same way, for $1\leq i\leq r$, we have, if $s\geq 0$,
$$\int_{\mathcal C_i}\exp(-s\|\pi(x)\|)\de\mu(x)\leq \int_{\mathcal C_i}\exp(-s(1+\varepsilon)^{-1}\alpha_i\|x\|)\de\mu(x)$$
and, if $s\leq 0$,
$$\int_{\mathcal C_i}\exp(-s\|\pi(x)\|)\de\mu(x)\leq \int_{\mathcal C_i}\exp(-s(1+\varepsilon)\alpha_i\|x\|)\de\mu(x).$$
This yields
$$\tau_{\mathcal D}\leq 
\max(a,\max_{1\leq i\leq r}(\alpha_i(1+\varepsilon)\tau_{\mathcal C_i}),\max_{1\leq i\leq r}(\alpha_i(1+\varepsilon)^{-1}\tau_{\mathcal C_i})),$$
hence
$$\psi_{\nu}(y)\leq \max(a\|y\|,\max_{1\leq i\leq r}(\|x_i\|(1+\varepsilon)\tau_{\mathcal C_i}),\max_{1\leq i\leq r}(\|x_i\|(1+\varepsilon)^{-1}\tau_{\mathcal C_i})).$$
In view of the assumption, we obtain
$$\psi_\nu(y)\leq \max((1+\varepsilon)\sup_{x\in E_y}\psi_\mu(x),(1+\varepsilon)^{-1}\sup_{x\in E_y}\psi_\mu(x))+\varepsilon(1+\varepsilon).$$
The result follows by letting $\varepsilon$ go to $0$.
\end{proof}

\section{Concave growth measures}
\label{concdiv}

As in \cite{Qui}, we say that a Radon measure $\mu$ on $E$ has concave growth if there exists $\alpha,\beta,\gamma>0$ such that, for every $x,y$ in $E$, one has
$$\mu(B(x+y,\alpha))\geq\gamma \mu(B(x,\beta))\mu(B(y,\beta)).$$
This notion actually does not depend on the choice of the norm.
We get from \cite[Th\'eor\`eme 3.2.1]{Qui}:

\begin{Prop} \label{concavegrth}
Let $\mu$ be a Radon measure with subexponential divergence and concave growth on $E$. The the indicator of growth $\psi_\mu$ is concave.
\end{Prop}

Thus, we can apply the formalism of Section \ref{hahnbanach}. This yields

\begin{Cor} \label{projconcave}
Let $F$ be a subspace of $E$ and $\pi:E\rightarrow G=E/F$ be the quotient map. Let $\mu$ be a Radon measure with subexponential divergence and concave growth on $E$ and suppose that $\psi_\mu(x)=-\infty$ for every $x\neq 0$ in $F$. Then, the measure 
$\nu=\pi_*\mu$ is finite on compact sets and has subexponential divergence. The function $\psi_\nu$ is concave; 
more precisely, for $y$ in $G$, we have
$$\psi_\nu(y)=\sup_{\substack{x\in E\\ \pi(x)=y}}\psi_\mu(x).$$
If $\Omega_\nu\subset G^*$ is the associated non empty closed convex subset, or every $\varphi$ in $G^*$, we have 
$$\int_G\exp(-\varphi(y))\de\nu(y)<\infty\Rightarrow \varphi\in\Omega_\nu$$
and
$$\varphi\in\overset{\circ}{\Omega}_\nu\Rightarrow \int_G\exp(-\varphi(y))\de\nu(y)<\infty.$$
\end{Cor}

\begin{proof} In view of \eqref{phiomega1} and \eqref{phiomega2}, it suffices to show that $\psi_\nu$ is concave. This is actually a direct consequence of Proposition \ref{projection}. Indeed, in view of this result, for $y$ in $G$, we have
$\psi_\nu(y)=\sup_{\substack{x\in E\\ \pi(x)=y}}\psi_\mu(x)$. Take $y,y'$ in $G$ and set $y''=y+y'$. Fix $x$ in $E$ with $\pi(x)=y$. 
For any $x''$ in $E$ with $\pi(x'')=y''$, we have $\pi(x''-x)=y'$, hence, since $\psi_\mu$ is concave,
$$\psi_\mu(x'')\leq\psi_\mu(x)+\psi_\mu(x''-x)\leq \psi_\nu(y)+\psi_\nu(y').$$
We get $\psi_\nu(y'')\leq \psi_\nu(y)+\psi_\nu(y')$ as required.
\end{proof}

\section{Multivariate power series}
\label{powerseries}

We now relate the previously introduced formalism to the study of the domain of convergence of multivariate power series in the spirit of analytic combinatorics (see \cite{Comt, FlajSedg, GoulJac}).

Let $d\geq 1$ and suppose we are given a family $f=(f_n)_{n\in\mathbb N^d}$ of complex numbers. We associate to $f$ the domain of absolute convergence $D_f$ of the associate power series in $\mathbb C^d$,
$$D_f=\{z\in\mathbb C^d|\sum_{n\in\mathbb N^d}|f_n z_1^{n_1}\cdots z_d^{n_d}|<\infty\}.$$
The set $D_f$ is stable under multiplication of the coordinates by complex numbers of modulus $1$, hence it is completely determined by the set $D_f\cap[0,\infty)^d$. We note that the latter set satisfies the following logarithmic convexity property. 

\begin{Lem} Let $z$ and $w$ be in $D_f\cap[0,\infty)^d$ and $0\leq s\leq 1$, then $z^sw^{(1-s)}$ also belongs to $D_f$, where we have written
$$z^sw^{(1-s)}=(z_1^sw_1^{(1-s)},\ldots,z_d^sw_d^{(1-s)}).$$
\end{Lem} 

\begin{proof} This is a direct consequence of H\"older inequality: we set $p=s^{-1}$ and $q=(1-s)^{-1}$ and we get 
\begin{multline*}\sum_{n\in\mathbb N^d}|f_n| (z_1^{s}w_1^{(1-s)})^{n_1}\cdots (z_d^{s}w_d^{(1-s)})^{n_d}\\ 
\leq \left(\sum_{n\in\mathbb N^d}|f_n| z_1^{n_1ps}\cdots z_d^{n_d ps}\right)^s
\left(\sum_{n\in\mathbb N^d}|f_n| w_1^{n_1q(1-s)}\cdots w_d^{n_dq(1-s)}\right)^{(1-s)}
\\=\left(\sum_{n\in\mathbb N^d}|f_n| z_1^{n_1}\cdots z_d^{n_d }\right)^s
\left(\sum_{n\in\mathbb N^d}|f_n| w_1^{n_1}\cdots w_d^{n_d}\right)^{(1-s)}<\infty.
\end{multline*}
\end{proof}

From this lemma, we deduce that the set 
$$\{(\log|z_1|,\ldots,\log|z_d|)|z\in D_f\cap (\mathbb C^*)^d\}\subset \mathbb R^d$$
is convex. We denote by $\Omega_f$ its closure. It is non empty as soon as 
 $D_f$ contains at least one point all of whose coordinates are $\neq 0$, which is to say
that there exists $R,C> 0$. with 
$$\sum_{\substack{n\in\mathbb N^d\\ n_1+\ldots +n_d=k}} f_{n}\leq CR^{-k}, \quad k\geq 0.$$

We can relate this formalism to the one of measures with subexponential divergence.

\begin{Lem} \label{dualdomconv}
Let $d\geq 1$ and $f=(f_n)_{n\in\mathbb N^d}$ be in $\mathbb C^{\mathbb N^d}$. Assume that 
$D_f\cap (\mathbb C^*)^d$ is non empty. Then, the Radon measure $\nu$ on $\mathbb R^d$ defined by
$$\nu=\sum_{n\in\mathbb N^d}|f_n|\delta_n$$
has subexponential divergence. Identify $\mathbb R^d$ with its dual space through the standard pairing. Then, 
the set $\Omega_f$ is the set of linear functionals $\varphi$ on $\mathbb R^d$ with $\varphi\geq\psi_\nu$.
\end{Lem}

When the measure $\nu$ has concave growth in the sense of Section \ref{concdiv}, Proposition \ref{concavegrth} says that the function $\psi_\nu$ is concave and therefore, by Lemma \ref{dualdomconv},
the convex set $\Omega_f$ and the function $\psi_\nu$ are dual to each other in the sense of Section \ref{hahnbanach}
(see \cite[Thm. 4.1]{GQS}).

\begin{proof}[Proof of Lemma \ref{dualdomconv}] The first part of the lemma directly follows from the definitions. For the second part, which is the reformulation of the definition of $\Omega_f$, we note that any $\varphi$ in $\{(\log|z_1|,\ldots,\log|z_d|)|z\in D_f\cap (\mathbb C^*)^d\}$, we have $\varphi\geq\psi_\nu$ in view of \eqref{phiomega2}. Conversely, the asumption implies that, for large $M>0$, the linear functional 
$\theta$ with coordinates $(M,\ldots,M)$ satisfy $\theta>\psi_\nu$ everywhere on $\mathbb R^d\smallsetminus\{0\}$, hence, if $\varphi\geq \psi_\nu$, then, for any $\varepsilon>0$, we have
$\varphi+\varepsilon\theta>\psi_\nu$ everywhere on $\mathbb R^d\smallsetminus\{0\}$ and the conclusion follows from \eqref{phiomega1}.
\end{proof}

\section{Languages and directed graphs} 
\label{secgraphs}

In this article, we deal with a multivariate counting problem for two special subclasses of the class of regular languages. Recall that a regular language is a language determined (accepted) by a finite automaton (as is explained for example in \cite{HU,Law}. The multivariate growth function of such a language is rational and can be computed for instance using Theorem 4.1 from \cite{GQS}. The latter article also provides a foundational tool for finding directional asymptotic of growth that is based on the theory developed in \cite{Qui}. An important role in this is played by the "ergodicity" of the language.

Here, we focus on two special types of regular languages, arising in the study of finite type subshifts and sofic subshifts. Not getting into details, we mention only here that finite type subshifts correspond to languages determined by directed graphs and that the set of edges plays the role of the alphabet. While a more general type of subshifts, the sofic systems, are determined by labelled finite directed graphs when different edges can be labelled by the same symbol. The "ergodicity" for such languages corresponds to irreducibility of the subshift. We call such languages "finite type" and "sofic"respectively. For more details on this, see \cite{LM}. Let now express things rigorously.

Fix a finite set $A$, 
which we will consider as an alphabet, and denote by $W$ a language written with the alphabet $A$, that is, 
$W$ is a subset of the set $A^*=\bigsqcup_{n\geq 0} A^n$ of all finite words. Denote by $W_n=W\cap A^n$ the set of words with length $n$ in $W$.

We associate to $W$ a counting measure $\nu_{W}$ on the space $\mathcal E=\mathbb R^A$ of all real valued functions on $A$ as follows. For any $n\geq 0$ and any $w=(b_1,\ldots,b_n)$ in $A^n$, we denote by $\mathscr P(w)$ the function on $A$ such that $\mathscr P(w)(a)$ is the number of occurences of the letter $a$ in the word $b$, that is,
$$\mathscr P(w)(a)=|\{1\leq k\leq n|b_k=a\}|.$$
Set $\nu_{W}$ to be the image measure of the counting measure of $W$ under the map $\mathscr P$, that is,
$$\nu_{W}=\sum_{w\in W}\delta_{\mathscr P(w)},$$
where $\delta_x$ stands for the Dirac mass at a point $x$.
Note that, as $A$ is finite, the measure $\nu_{W}$ has subexponential divergence. We denote its indicator of growth by $\psi_{W}$.

For languages with a reasonable connectivity property, the measure $\nu_{W}$ has concave growth.

\begin{Lem} \label{irreducibleconcave}
Assume that the language $W$ is irreducible in the following sense: there exists an integer $p\geq 0$ such that, for every $n,m\geq 0$, and every $w\in W_n$, $w'\in W_m$, there exists $0\leq q\leq p$ and $w''\in W_{n+m+q}$ such that $w$ is the beginning of $w''$ and $w'$ is the end of $w''$. Then, the function $\psi_{W}$ is concave. 
\end{Lem}

The proof is a direct consequence of the assumption and the concavity statement for measures with concave growth, namely, the Proposition \ref{concavegrth}.

%
%

We will now study languages that are produced by a directed graph. By a directed graph, we shall mean the data of a list
$\mathcal G=(Q,A,\sigma,\gamma)$ where $Q$ and $A$ are finite sets and $\sigma$ and $\gamma$ are maps from $A$ to $Q$. We shall think to $Q$ as being a set of states or vertices and to $A$ as being a set of arrows between the states. The source of an arrow $a$ is $\sigma(a)$ and its goal is $\gamma(a)$. See Figure \ref{fibonaccifig} and Figure \ref{cyclicfig} for examples.

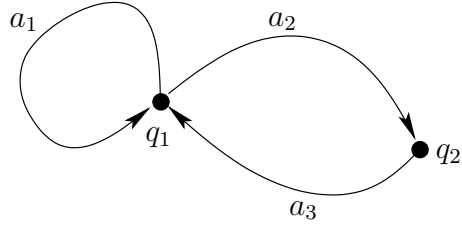
\begin{figure}\begin{center}
\begingroup%
  \makeatletter%
  \providecommand\color[2][]{%
    \errmessage{(Inkscape) Color is used for the text in Inkscape, but the package 'color.sty' is not loaded}%
    \renewcommand\color[2][]{}%
  }%
  \providecommand\transparent[1]{%
    \errmessage{(Inkscape) Transparency is used (non-zero) for the text in Inkscape, but the package 'transparent.sty' is not loaded}%
    \renewcommand\transparent[1]{}%
  }%
  \providecommand\rotatebox[2]{#2}%
  \newcommand*\fsize{\dimexpr\f@size pt\relax}%
  \newcommand*\lineheight[1]{\fontsize{\fsize}{#1\fsize}\selectfont}%
  \ifx\svgwidth\undefined%
    \setlength{\unitlength}{197.25bp}%
    \ifx\svgscale\undefined%
      \relax%
    \else%
      \setlength{\unitlength}{\unitlength * \real{\svgscale}}%
    \fi%
  \else%
    \setlength{\unitlength}{\svgwidth}%
  \fi%
  \global\let\svgwidth\undefined%
  \global\let\svgscale\undefined%
  \makeatother%
  \begin{picture}(1,0.47528517)%
    \lineheight{1}%
    \setlength\tabcolsep{0pt}%
    \put(0,0){\includegraphics[width=\unitlength,page=1]{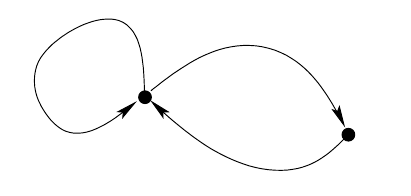}}%
    \put(0.32317608,0.16412644){\color[rgb]{0,0,0}\makebox(0,0)[lt]{\lineheight{1.25}\smash{\begin{tabular}[t]{l}$q_1$\end{tabular}}}}%
    \put(0.87780719,0.13158901){\color[rgb]{0,0,0}\makebox(0,0)[lt]{\lineheight{1.25}\smash{\begin{tabular}[t]{l}$q_2$\end{tabular}}}}%
    \put(0.06330141,0.38266178){\color[rgb]{0,0,0}\makebox(0,0)[lt]{\lineheight{1.25}\smash{\begin{tabular}[t]{l}$a_1$\end{tabular}}}}%
    \put(0.55777007,0.38450641){\color[rgb]{0,0,0}\makebox(0,0)[lt]{\lineheight{1.25}\smash{\begin{tabular}[t]{l}$a_2$\end{tabular}}}}%
    \put(0.59768416,0.02313303){\color[rgb]{0,0,0}\makebox(0,0)[lt]{\lineheight{1.25}\smash{\begin{tabular}[t]{l}$a_3$\end{tabular}}}}%
  \end{picture}%
\endgroup%

\caption{The Fibonacci graph}
\label{fibonaccifig}
\end{center}\end{figure}

Given such a data, we define a language $W$ on the alphabet $A$ which describes the set of admissible paths on the graph. Write 
$W=\bigsqcup_{n\geq 0}W_n$, where, for $n\geq 0$,
$$W_n=\{(a_1,\ldots,a_n)\in A^n|\forall 1\leq k\leq n-1\quad \gamma(a_k)=\sigma(a_{k+1})\},$$
that is, $W_n$ is the set of paths with length $n$ in the directed graph $\mathcal G$.
Then the language is irreducible if and only if the graph is connected that is, if and only if given $q,q'$ in $Q$, we may find $n\geq 0$ such that the set of paths of length $n$ from $q$ to $q'$,
$$W_n^{q,q'}=\{(a_1,\ldots,a_n)\in W_n|\sigma(a_1)=q,\gamma(a_n)=q'\},$$
is not empty. In other words, there exists t least one path from $q$ to $q'$.

Assume this holds. Then, if $\mathcal E=\mathbb R^A$ is the space of real-valued functions on $A$, 
by Lemma \ref{irreducibleconcave}, the function 
$\psi_{W}$ on $\mathcal E$ is concave. We shall now write $\psi_{\mathcal G}$ for $\psi_{W}$.

We will give a direct description of the set $\Omega_\mathcal G\subset\mathcal E^*$ which is dual to the function $\psi_{\mathcal G}$. 
Identify $\mathcal E^*$ with $\mathcal E$ by using the standard inner product 
$$\langle f,g\rangle=\sum_{a\in A}f(a)g(a),\quad f,g\in\mathcal E.$$

Denote the space of real valued functions on $Q$ by $V$, that is, $\mathcal V=\mathbb R^Q$, which we also identify to its dual space by means of the standard scalar product. 
Then, to every $\theta$ in $\mathcal E$, we associate a linear operator $L_\theta$ on $\mathcal V$ by the formula
$$(L_\theta f)(q)=\sum_{\substack{a\in A\\ \sigma(a)=q}}e^{-\theta(a)}f(\gamma(a)),\quad q\in Q,f\in\mathcal V.$$
This operator is an analogue of the transfer operator of hyperbolic dynamics (see \cite{PP}); it can also be thought of as a weighted Laplace operator on the directed graph $\mathcal G$.
Note that this operator is non-negative, or equivalently Perron-Frobe\-nius.

\begin{Prop} \label{omegatransfer}
Let $\mathcal G=(Q,A,\sigma,\gamma)$ be a connected directed graph.
Then, the set $\Omega_\mathcal G$ is the set of $\theta$ in $\mathcal E$ such that the operator $L_\theta$ has spectral radius $\leq 1$ in $\mathcal V$. 
\end{Prop}

To prove this, we will need to study more closely the operators $L_\theta$. First, we recall some basic facts about Perron-Frobenius operators (see \cite{Meyer}). We state these facts in an abstract form. Recall that a closed convex cone $\mathcal C\subset \mathcal E$ is said to be proper if $\mathcal C\cap(-\mathcal C)=\{0\}$, that is, $\mathcal C$ does not contain any vector line.

\begin{Lem} \label{PF}
Let $H$ be a finite-dimensional real vector space and $\mathcal C\subset H$ be a proper closed convex cone with non empty interior. Let $T:H\rightarrow H$ be a linear map with $T\mathcal C\subset \mathcal C$. Then, the spectral radius $\lambda$ of $T$ is an eigenvalue of $T$ and there is an associated eigenvector in $\mathcal C\smallsetminus\{0\}$. If there exists $p\geq 1$ such that the operator $\sum_{j=0}^p T^j$ maps $\mathcal C\smallsetminus\{0\}$ into the interior $\overset{\circ}{\mathcal C}$ of $\mathcal C$, then $\lambda$ has multiplicity one and the associated eigenline is spanned by an element of $\overset{\circ}{\mathcal C}$. 
\end{Lem}

Recall that saying that $\lambda$ has multiplicity $1$ is saying that $\lambda$ is a simple root of the characteristic polynomial of $T$.

For $\theta$ in $\mathcal E$, we denote by $\lambda(\theta)$ the spectral radius of $L_\theta$ in $\mathcal V$. 
By Lemma \ref{PF}, up to a scalar multiple,
there exists a unique non zero function $f_\theta$ in $\mathcal V$ with $L_\theta f_\theta=\lambda(\theta)f_\theta$ and we may choose $f_\theta$ to be positive.
In the same way, applying Lemma \ref{PF} to the adjoint operator $L_\theta^*$ (acting on the dual space of $\mathcal V$, which we have identified with $\mathcal V$) yields a unique positive linear functional $\varphi_\theta$ such that $L_\theta^*\varphi_\theta= \lambda(\theta)\varphi_\theta$ and $\langle\varphi_\theta,{\bf 1}\rangle=1$, where ${\bf 1}$ is the constant function with value $1$ on $Q$. 
Since we have identified $\mathcal V$ and its dual space, we consider $\varphi_\theta$ as a positive function on $\mathcal V$.
We normalize $f_\theta$ by asking that $\langle\varphi_\theta,f_\theta\rangle=1$.
Below, we will show that the function $\log\lambda(\theta)$ is a convex function of $\theta$.
We will first describe the other eigenvalues of $L_\theta$ with modulus $\lambda(\theta)$. They are multiples of $\lambda(\theta)$ by roots of unity.

\begin{Lem} \label{modulus1} Let $\mathcal G=(Q,A,\sigma,\gamma)$ be a connected directed graph.
There exists $p\geq 1$ and a complex valued function $u$ on $Q$, with values of modulus $1$, with the following property. 
For every $a$ in $A$, one has
$u(\gamma(a))=e^{2\pi i p^{-1}}u(\sigma(a))$. Then, for every $\theta$ in $\mathcal E$, the eigenvalues of modulus $1$ of $L_\theta$ are the $e^{2\pi i kp^{-1}}$, $0\leq k\leq p-1$; they have multiplicity one and the associated eigenspaces in $\mathcal V_\mathbb C$ are the lines spanned by $u^k f_\theta$, $0\leq k\leq p-1$.
\end{Lem}

We wrote $\mathcal V_\mathbb C$ for the complexification of $\mathcal V$, that is, the space of all complex valued functions on $Q$.

On Figure \ref{cyclicfig}, we have represented a directed graph for which $p=2$. Indeed, the space of states is partitioned into the set of black states and the set of white states. Every arrow with a black source has a white goal; every arrow with a white source has a black goal. 
The function $u$ can be chosen to take the value $1$ on black states and $-1$ on white states. 

\begin{figure}\begin{center}
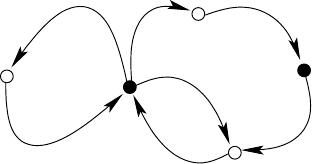
\caption{A directed graph with $p=2$}
\label{cyclicfig}
\end{center}\end{figure}

\begin{proof}[Proof of Lemma \ref{modulus1}] Let $\mathbb T\subset\mathbb C$ denote the set of complex numbers with modulus $1$. 
For $z$ in $\mathbb T$, define $U_z\subset \mathcal V_\mathbb C$ as the space of all functions $f$ such that, 
for every $a$ in $A$, one has
for $f(\gamma(a))=zf(\sigma(a))$. Set $G\subset\mathbb T$ to be the set of all $z$ such that $U_z$ is non zero. We claim that $G$ is a finite multiplicative subgroup. 

Indeed, one easily shows that the vector subspaces $U_z$, $z\in G$, are linearly independent. Besides, 
as $\mathcal G$ is connected, for $z$ in $G$, 
the modulus of a function $f$ in $U_z$ is necessarily constant.
Thus, if $z,z'$ are in $G$, then $U_{zz'}$ is non zero and equal to the set $U_zU_{z'}$ of functions that are products of a function in $U_z$ and a function in $U_{z'}$. Thus, $G$ is a finite subgroup of $\mathbb T$. Hence, it is of the form 
$$G=\{e^{2\pi i kp^{-1}}|0\leq k\leq p-1\}$$
for some $p\geq 1$.

Moreover, still as $\mathcal G$ is connected, $U_1$ is the line of constant functions, hence all $U_z$, $z\in G$, have dimension $1$. We choose a function $u$ with constant modulus $1$ in $U_{z}$ for $z=e^{2\pi i p^{-1}}$. Then, for $0\leq k\leq p-1$, the space $U_{z^k}$ is the line spanned by $u^k$.

Let now $\theta$ be in $\mathcal E$. For $0\leq k\leq p-1$, we have $L_\theta(u^k f_\theta)=\lambda(\theta)e^{2\pi i kp^{-1}}f_\theta$. 
The converse statement that all eigenvalues with modulus $\lambda_\theta$ of $L_\theta$ are obtained in this way will follow from an application of the maximum principle. 

Indeed, let $z$ be in $\mathbb T$ and $f\neq 0$ be in $\mathcal V_\mathbb C$ with $L_\theta f=z\lambda(\theta)f_\theta$. We set $g=f_\theta^{-1} f$.
For $q$ in $Q$, we have
$$\sum_{\substack{a\in A\\ \sigma(a)=q}}e^{-\theta(a)}\frac{f_\theta(\gamma(a))}{\lambda(\theta) f_\theta(q)}g(\gamma(a))
=\frac{1}{\lambda(\theta) f_\theta(q)}L_\theta(f)(q)=zg(q).$$
Set $Q'$ to be the set of $q$ in $Q$ such that $|g(q)|$ is maximal. For $q$ in $Q'$, we have
$$\left|\sum_{\substack{a\in A\\ \sigma(a)=q}}e^{-\theta(a)}\frac{f_\theta(\gamma(a))}{\lambda(\theta) f_\theta(q)}g(\gamma(a))\right|=|g(q)|$$
and all the $g(\gamma(a))$ have modulus $\leq g(q)$. Thus, since 
$$\sum_{\substack{a\in A\\ \sigma(a)=q}}e^{-\theta(a)}\frac{f_\theta(\gamma(a))}{\lambda(\theta) f_\theta(q)}=1,$$
as the disks of $\mathbb C$ are strictly convex,
for any $a$ with $\sigma(a)=q$, we have $g(\gamma(a))=zg(q)$ and in particular, $\gamma(a)$ belongs to $Q'$. 
Since $\mathcal G$ is connected, we obtain that $Q'=Q$ and that $g$ belongs to $U_z$. Thus, $g$ belongs to $\mathbb C u^k f_\theta$ for some 
$0\leq k\leq p-1$ and hence, the eigenvalues of modulus $\lambda_\theta$ of $L_\theta$ are the $e^{2\pi i kp^{-1}}$,
$0\leq k\leq p-1$. 

Since, for any such $k$, multiplication by $u^k$ conjugates the operators 
$L_\theta-\lambda_\theta e^{2\pi i kp^{-1}}$ and $e^{2\pi i kp^{-1}}(L_\theta-\lambda_\theta)$, the eigenvalue $e^{2\pi i kp^{-1}}\lambda_\theta$ has multiplicity $1$.
\end{proof}

In the sequel, we let $p$ be and $u$ be as in Lemma \ref{modulus1}.
As a Corollary of the proof, we get another definition of $p$. For $n\geq 0$ and $q,q'$ in $Q$, recall that $W_n^{q,q'}$ is the set of paths of length $n$ from $q$ to $q'$.

\begin{Cor} \label{modulus12}
Let $\mathcal G=(Q,A,\sigma,\gamma)$ be a connected directed graph.
Then, for every $q$ in $Q$, $p$ is the greatest common divisor of the set of integers
$$\{n\in\mathbb N| W_n^{q,q}\neq\emptyset\}.$$
\end{Cor}

\begin{proof} Fix $q_0$ in $A$ and let $d$ be the greatest common divisor of the set of $n$ in $\mathbb N$ with 
$W_n^{q_0,q_0}\neq\emptyset$. For every such $n$, we have $u(q_0)^n=1$, hence $u(q_0)^d=1$. This gives that $p$ divides $d$. Conversely, we will use the definition of $p$ in the proof of Lemma \ref{modulus1} to show that $d$ divides $p$.

For $q$ in $Q$ define $f(q)=e^{-2i\pi kd^{-1}}$ where $k$ is the least integer such that $W_k^{q,q_0}\neq\emptyset$. 
As $\mathcal G$ is connected, this function is well defined. Now, let $\ell\geq k$ be such that $W_\ell^{q,q_0}\neq\emptyset$. We claim that we also have $f(q)=e^{-2i\pi \ell d^{-1}}$. Indeed, if we set $n=\ell-k$, we have $W_n^{q,q}\neq\emptyset$, hence $d$ divides $n$, so that 
$e^{-2i\pi \ell d^{-1}}=e^{-2i\pi k d^{-1}}$ as required. In particular, pick $a$ is in $A$ with $\gamma(a)=q$ and set $\sigma(a)=q'$.
We have 
$W_{k+1}^{q',q_0}\neq\emptyset$, which gives 
$$f(q')=e^{-2i\pi (k+1)d^{-1}}=e^{-2i\pi d^{-1}}f(q).$$
Thus, by the definition of $p$, $p$ divides $d$ as required.
\end{proof}

These constructions allow us to describe the asymptotic behaviour as $n\rightarrow\infty$ of $L_\theta^n f$, for $f$ positive. For $q$ in $Q$ and $j$ in $\mathbb Z/p\mathbb Z$, denote by $Q_j^q$ the set of $q'$ in $Q$ for which there exists $n\geq 0$ with $n\in j+p\mathbb Z$ such that the set $W_n^{q,q'}$ is not empty. Equivalently, $Q_j^q$ is the set of $q'$ such that $u(q')=e^{2\pi ijp^{-1}}u(q)$ where $u$ is as in Lemma \ref{modulus1}.
For $n$ in $\mathbb Z$, write $Q_n^q=Q_j^q$ where $j$ is the class of $n$ in $\mathbb Z/p\mathbb Z$.

\begin{Cor} \label{asymptotic1}
Let $\mathcal G=(Q,A,\sigma,\gamma)$ be a connected directed graph.
Let $\theta$ be in $\mathcal E$. There exists 
$\varepsilon,C>0$ such that, for every $f$ in $\mathcal E$, $q$ in $Q$ and $n\geq 0$, one has
\begin{equation}\label{asymptotic}
|L_\theta^nf(q)-p\lambda(\theta)^n\langle \varphi_\theta,f{\bf 1}_{Q^q_n}\rangle f_\theta(q) |\leq C(\lambda(\theta)-\varepsilon)^n\|f\|.\end{equation}
\end{Cor}

In the above, we have denoted by ${\bf 1}_{Q^q_n}$ the characteristic function of the set $Q^q_n\subset Q$. 

\begin{proof} 
In view of Lemma \ref{modulus1}, all the eigenvalues of the operator 
$$f\mapsto f-
\sum_{k\in\mathbb Z/p\mathbb Z}\langle \varphi_\theta,u^{-k}f\rangle u^k f_\theta$$
have modulus $<\lambda(\theta)$ (note that this operator is actually real).
Thus, we can find $\varepsilon,C>0$ such that, for $f$ in $\mathcal E$ and $n$ in $\mathbb N$, we have 
\begin{equation}\label{asymptotic2}\|L_\theta^nf-\lambda(\theta)^n \Gamma_{\theta,n}f\|
\leq C(\lambda(\theta)-\varepsilon)^n\|f\|,\end{equation}
where, for $j$ in $\mathbb Z/p\mathbb Z$, we have set 
$$\Gamma_{\theta,j}f=
\sum_{k\in\mathbb Z/p\mathbb Z}e^{2\pi i jkp^{-1}}\langle \varphi_\theta,u^{-k}f\rangle u^k f_\theta.$$

To conclude, we will now compute these operators in a different way. 
This will be an application of the Fourier inversion formula on $\mathbb Z/p\mathbb Z$.
Fix $q$ in $Q$ and $j$ in $\mathbb Z/p\mathbb Z$.
For $k$ in $\mathbb Z/p\mathbb Z$, we have
$$u^{-k}f=u^{-k}(q)\sum_{\ell\in\mathbb Z/p\mathbb Z}e^{-2\pi i \ell kp^{-1}}f{\bf 1}_{Q^q_\ell},$$
which yields
\begin{align*}\Gamma_{\theta,j}f(q)&=
\sum_{k\in\mathbb Z/p\mathbb Z}e^{2\pi i jkp^{-1}}\sum_{\ell\in\mathbb Z/p\mathbb Z}e^{-2\pi i \ell kp^{-1}}
\langle \varphi_\theta,f{\bf 1}_{Q^q_\ell}\rangle f_\theta(q)\\
&=\sum_{\ell\in\mathbb Z/p\mathbb Z}\langle \varphi_\theta,f{\bf 1}_{Q^q_\ell}\rangle f_\theta(q)
\sum_{k\in\mathbb Z/p\mathbb Z} e^{2\pi i (j-\ell)kp^{-1}}.\end{align*}
For $\ell\neq j$, we have $\sum_{k\in\mathbb Z/p\mathbb Z} e^{2\pi i (j-\ell)kp^{-1}}=0$ and we get
$$\Gamma_{\theta,j}f(q)=p\langle \varphi_\theta,f{\bf 1}_{Q^q_j}\rangle f_\theta(q).$$
This together with \eqref{asymptotic2} implies \eqref{asymptotic}.
\end{proof}

\begin{Cor} \label{lambdaconvex}
Let $\mathcal G=(Q,A,\sigma,\gamma)$ be a connected directed graph.
Then, the function $\theta\mapsto\log\lambda(\theta)$ is convex on $\mathcal E$.
\end{Cor}

The proof uses

\begin{Lem} \label{logconvex}
Let $r\geq 1$ and $\alpha=(\alpha_1,\ldots,\alpha_r)$ and $\beta=(\beta_1,\ldots,\beta_r)$ be in $\mathbb b R^r$. Then, the function 
$$h_{\alpha,\beta}:x\mapsto\log\left(\sum_{i=1}^re^{\alpha_i x+\beta_i}\right)$$
is convex on $\mathbb R$.
\end{Lem}

\begin{proof} Since $h_{\alpha,\beta}$ is smooth, it suffices to check that its second deriative is everywhere nonnegative. Indeed, for $x$ in $\mathbb R$, we have 
$$h_{\alpha,\beta}'(x)=\frac{\sum_{i=1}^r\alpha_ie^{\alpha_i x+\beta_i}}{\sum_{i=1}^re^{\alpha_i x+\beta_i}}$$
and
$$h_{\alpha,\beta}''(x)=\frac{\left(\sum_{i=1}^r\alpha_i^2e^{\alpha_i x+\beta_i}\right)\left(\sum_{i=1}^re^{\alpha_i x+\beta_i}\right)-
\left(\sum_{i=1}^r\alpha_ie^{\alpha_i x+\beta_i}\right)^2}
{\left(\sum_{i=1}^re^{\alpha_i x+\beta_i}\right)^2}.$$
By Cauchy Schwarz inequality, the denominator of the latter fraction is always $\geq 0$. The Lemma follows.
\end{proof}

\begin{proof}[Proof of Corollary \ref{lambdaconvex}] We fix a positive function $f$ in $\mathcal V$ and a positive linear functional $\rho$ in $\mathcal V^*$. By Corollary \ref{asymptotic1}, for $\theta$ in $\mathcal E$, we have 
$$\log\lambda(\theta)=\lim_{n\rightarrow\infty}\frac{1}{n}\log \langle \rho,L_\theta^n f\rangle.$$
In view of Lemma \ref{logconvex} and of the definition of $L_\theta$, for any $n\geq 0$, the function 
$\theta\mapsto\log \langle \rho,L_\theta^n f\rangle$ is convex. The conclusion follows.
\end{proof}

We can now conclude.

\begin{proof}[Proof of Proposition \ref{omegatransfer}] For $q$ in $A$ and $n\geq 0$, define $W_n^q\subset W_n$ as the set of 
all words $w=(w_1,\ldots,w_n)$ such that $\sigma(w_1)=q$. Set also $W^q=\bigcup_{n\geq 0}W_n^q$.
We still let $\mathscr P:W\rightarrow \mathcal E$ be the natural map that counts the number of occurences of a given letter in a word.
Then, a direct computation gives, for $\theta$ in $\mathcal E$, 
$$L^n_\theta f(q)=\sum_{w\in W^q_n}e^{-\langle\theta,\mathscr P(w)\rangle}f(\gamma(w_n)),\quad f\in \mathcal V,n\geq 0,q\in Q,$$
where $w_n$ stands for the last letter on the right of the word $w$ (and by convention, $\gamma(w_n)=q$ if $n=0$).
We get 
\begin{multline*}
\sum_{n=0}^\infty\langle {\bf 1},L_\theta^n{\bf 1}\rangle=
\sum_{n=0}^\infty\sum_{q\in Q} L_\theta^n{\bf 1}(q)=\sum_{n=0}^\infty\sum_{q\in Q}\sum_{w\in W^q_n}e^{-\langle\theta,\mathscr P(w)\rangle}\\
=\sum_{w\in W} e^{-\langle\theta,\mathscr P(w)\rangle}.\end{multline*}
Therefore, by Corollary \ref{asymptotic1},
we have $\sum_{w\in W} e^{-\langle\theta,\mathscr P(w)\rangle}<\infty$ if and only if $\lambda(\theta)<1$.
By Lemma \ref{lambdaconvex}, the function $\log\lambda$ is convex on $\mathcal E$, hence it is continuous and the set where it is $<0$ is open. Thus, in view of Lemma \ref{interior} and \eqref{phiomega1} and \eqref{phiomega2}, we obtain
$$\theta\in \overset{\circ}{\Omega}_{\mathcal G}\Leftrightarrow \lambda(\theta)<1.$$
Still as $\log\lambda$ is convex, the set $\{\theta\in \mathcal E|\lambda(\theta)<1\}$ is the interior of the set 
$\{\theta\in \mathcal E|\lambda(\theta)\leq 1\}$. As $A$ is finite, the set $\overset{\circ}{\Omega}_{\mathcal G}$ is non empty, hence both closed convex sets 
$\Omega_{\mathcal G}$ and $\{\theta\in \mathcal E^*|\lambda(\theta)\leq 1\}$ are the closures of their interior and are therefore equal to each other.
\end{proof}

\begin{Ex} \label{fibonacci1}
We use the characterization of Proposition \ref{omegatransfer} to compute the objects in a simple example, 
which is called the Fibonacci graph (see Figure \ref{fibonaccifig}).

In this example, the set $A$ is written as $\{a_1,a_2,a_3\}$ and we identify $\mathcal E$ with $\mathbb R^3$. In the same way, the set $Q$ is written as $\{q_1,q_2\}$ and we identify $\mathcal V$ with $\mathbb R^2$. Note that, as the arrow $a_1$ has source and goal the state $q_1$, we have $p=1$ in view of Corollary \ref{modulus12}.
For $\theta=(\theta_1,\theta_2,\theta_3)$, the matrix of the operator $L_\theta$ is
$$L_\theta=\begin{pmatrix}e^{-\theta_1}&e^{-\theta_2}\\ e^{-\theta_3}&0\end{pmatrix}.$$
The eigenvalues of this matrix are the roots of the equation
$$\chi_\theta(x)=x^2-e^{-\theta_1}x-e^{-\theta_2-\theta_3}=0.$$
This equation has exactly one positive root, which is 
$$\lambda(\theta)=\frac{1}{2}(e^{-\theta_1}+\sqrt{e^{-2\theta_1}+4e^{-\theta_2-\theta_3}}).$$
As $\chi_\theta(x)\geq 0$ for $x$ large, saying that the positive root $\lambda(\theta)$ is $\leq 1$ amounts to saying that one has
$\chi_\theta(1)\geq 0$. From Proposition \ref{omegatransfer}, we get
$$\Omega_\mathcal G=\{\theta\in \mathbb R^3|e^{-\theta_1}+e^{-\theta_2-\theta_3}\leq 1\}.$$
This allows to determine the function $\psi_\mathcal G$. Indeed, as $\psi_\mathcal G$ is concave, we get from \eqref{relpsiomega}, 
for $x$ in $\mathcal E$,
$$\psi_\mathcal G(x)=\inf_{\theta\in\Omega_\mathcal G}\theta(x).$$
By using the formula in Example \ref{freepsi} together with Lemma \ref{concavetransport}, we obtain, for $x$ in $\mathbb R^3$,
\begin{align*}\psi_\mathcal G(x)&=x_1\log\frac{x_1+x_2}{x_1}+x_2\log\frac{x_1+x_2}{x_2}&&x_1,x_2,x_3\geq 0,x_2=x_3\\
&=-\infty&&\mbox{else}.\end{align*}
%
\end{Ex}

\section{Global asymptotic estimates}
\label{secglobalcounting}

We continue the study of asymptotics of languages defined by directed graphs and we keep the notation of the previous Section. We will now use the previously introduced formalism to establish an estimate of the number of words of length $n$. 

Say that the connected directed graph $\mathcal G=(Q,A,\sigma,\gamma)$ is cyclic if, for every $q$ in $Q$, there exists a unique $a$ in $A$ with $\sigma(a)=q$. Then, the sets $A$ and $Q$ may be identified with $\mathbb Z/p\mathbb Z$ while the source map $\sigma$ is the identity map and the goal map $\gamma$ is the map $j\mapsto j+1$.

For a non cyclic connected directed graph, the spectral radius $\lambda(0)$ of the operator $L_0$ is $>1$.

\begin{Lem} \label{purcycle}
Let $\mathcal G=(Q,A,\sigma,\gamma)$ be a connected directed graph. Then $\lambda(0)\geq 1$ with equality if and only if 
$\mathcal G$ is cyclic. 
\end{Lem}

\begin{proof} 
Clearly, if $\mathcal G$ is cyclic, we have $\lambda(0)=1$. Conversely, suppose $\mathcal G$ is not cyclic. Then, there exists $q$ in $Q$ and $a\neq a'$ in $A$ with $\sigma(a)=\sigma(a')=q$. As $\mathcal G$ is connected, we may find $k\geq 0$ with
$W_k^{\gamma(a),q}\neq\emptyset$ and $W_k^{\gamma(a'),q}\neq\emptyset$.
Then, we have $L^{k+1}_{0}({\bf 1}_q)(q)\geq 2$, hence, for any $n\geq 0$, $L^{n(k+1)}_{0}({\bf 1}_q)(q)\geq 2^n$.
Thus, $L_0$ has spectral radius $2^{\frac{1}{k+1}}$.
\end{proof}

Set
$$
\delta_\mathcal G=\log\lambda(0).$$
This number will describe the growth rate of the sets $W_n^q$.

\begin{Ex}\label{fibonacci2} For the Fibonacci graph of Example \ref{fibonacci1}, we compute $\delta_{\mathcal G}$. By definition, we have
$\lambda(\delta_{\mathcal G}{\bf 1})=1$ which amounts to $1-e^{-\delta_\mathcal G}-e^{-2\delta_\mathcal G}=0$ or equivalently to 
$e^{2\delta_\mathcal G}=e^{\delta_\mathcal G}+1$. Thus, $e^{\delta_\mathcal G}$ is the golden ratio 
$\frac{1+\sqrt{5}}{2}$.
\end{Ex}

Recall that, for $n\geq 0$ and $q$ in $Q$, $W_n^q$ stands for the set of words $w=(w_1,\ldots,w_n)$ with length $n$ in $W$ such that $\sigma(w_1)=q$. We will estimate the size of these sets. 
If $(a_n)_{n\geq 0}$ and $(b_n)_{n\geq 0}$ are sequences of positive real numbers, we write $a_n \underset{n\rightarrow\infty}{\sim}b_n$ to say that $\frac{a_n}{b_n}\td{n}{\infty}1$.

\begin{Prop} \label{globequiv}
Let $\mathcal G=(Q,A,\sigma,\gamma)$ be a connected directed graph.
Then, if $p=1$, for $q$ in $Q$, 
$$|W_n^q|\underset{n\rightarrow\infty}{\sim}e^{\delta_\mathcal G n}
\langle\varphi_{\delta_{\mathcal G}\bf 1},{\bf 1}\rangle f_{\delta_{\mathcal G}\bf 1}(q).
$$
In general, for $q$ in $Q$, 
$$|W_n^q|\underset{n\rightarrow\infty}{\sim}
p e^{\delta_{\mathcal G} n} \langle\varphi_{\delta_{\mathcal G}\bf 1},{\bf 1}_{Q_{n}^q}\rangle f_{\delta_{\mathcal G}\bf 1}(q).$$
\end{Prop}

\begin{proof}
We fix $q$ in $Q$. By definition, for $n\geq 0$, we have 
$$L_{\delta_{\mathcal G}{\bf 1}}^n{\bf 1}(q)=e^{-\delta_{\mathcal G}n}|W_n^q|.$$
As $\lambda(\delta_{\mathcal G}{\bf 1})=1$, the conclusion follows from Corollary \ref{asymptotic1}.
\end{proof}

\section{The function $\lambda$}
\label{seclambda}

We now aim at studying the number of elements $w$ in $W_n$ such that $\mathscr P(w)$ takes a prescribed value. This is closely related to large deviation theory of random walks on Euclidean spaces and will require us to have a better understanding of the function $\lambda$. This function will play the role played by the Laplace transform of the law of the random walk in large deviation theory.

The fine study of the function $\lambda$ requires us to introduce a new subspace of $\mathcal E=\mathbb R^A$. 
For $f$ in $\mathcal V=\mathbb R^Q$, we let $\nabla f$ be the function on $A$ defined by, for $a$ in $A$,
$$\nabla f(a)=f(\gamma(a))-f(\sigma(a)).$$
The space $\nabla \mathcal V\subset \mathcal E$ will play a role in the statements of the next results. 

\begin{Lem} \label{injective} Let $\mathcal G=(Q,A,\sigma,\gamma)$ be a connected directed graph.
For $f$ in $\mathcal V$ and $c$ in $\mathbb R$, 
we have $\nabla f=c{\bf 1}$ if and only if $c=0$ and $f$ is constant.
\end{Lem}

\begin{proof} Suppose $\nabla f=c{\bf 1}$. This amounts to saying that, for every $a$ in $A$, we have $f(\gamma(a))-f(\sigma(a))=c$. Fix $q$ in $Q$. As $\mathcal G$ is connected, we may find $n\geq 1$ with $W_n^{q,q}\neq\emptyset$. We get $0=f(q)-f(q)=nc$, hence $c=0$. In particular, 
for every $a$ in $A$, we have $f(\sigma(a))=f(\gamma(a))$. Using again the connectedness of $\mathcal G$, we get that $f$ is constant as required. 
\end{proof}

\begin{Ex} \label{fibonacci3}
For the Fibonacci graph of Example \ref{fibonacci1} and Example \ref{fibonacci2}, the space $\nabla \mathcal V\subset \mathcal E$ is the space of all $\theta=(\theta_1,\theta_2,\theta_3)$ in $\mathbb R^3$ with $\theta_1=0$ and $\theta_2+\theta_3=0$. Note that this space is the orthogonal subspace to the space spanned by the domain of definition of the function $\psi_{\mathcal G}$ (see Example \ref{fibonacci1}). This is general phenomenon, which will be explained in Corollary \ref{psiregular} below.
\end{Ex}

\begin{Prop} \label{lambdaderivative} Let $\mathcal G=(Q,A,\sigma,\gamma)$ be a connected directed graph.
Then the function $\lambda$ is real analytic on $\mathcal E$. 

For $\theta$ in $\mathcal E$, the derivative of $\lambda$ at $\theta$ is given by
\begin{equation}\label{derivative6}
\de_\theta\lambda(\xi)=-\sum_{a\in A}\varphi_\theta(\sigma(a))e^{-\theta(a)}\xi(a)f_\theta(\gamma(a)),\quad\xi \in \mathcal E.\end{equation}

The second derivative $\de^2_\theta\log\lambda$ of $\log\lambda$ at $\theta$ is a non-negative symmetric bilinear form on $\mathcal E$ whose kernel is the space $\mathbb R{\bf 1}\oplus\nabla \mathcal V$ of functions that may be written as the sum of a constant function and an element of $\nabla \mathcal V$.
Additionally, for $c$ in $\mathbb R$ and $\varphi$ in $\mathcal V$, we have
\begin{equation}\label{invariance}\lambda(\theta+c+\nabla\varphi)=e^{-c}\lambda(\theta).\end{equation}
\end{Prop}

The proof of this result will last until the end of the section.

Recall from Section \ref{secgraphs} that $f_\theta$ is an eigenvector associated to the eigenvalue $\lambda(\theta)$ for $L_\theta$ acting on $\mathcal V$ and that $\varphi_\theta$ is an eigenvector associated to the eigenvalue $\lambda(\theta)$ for the adjoint operator $L_\theta^*$. They are normalized by the relations $\langle \varphi_\theta,{\bf 1}\rangle=1=\langle \varphi_\theta,f_\theta\rangle$. 

Note that Proposition \ref{lambdaderivative} implies that the kernel of $\de^2_\theta\log\lambda$ does not depend on $\theta$. In view of Lemma \ref{injective}, its dimension is the dimension of $\mathcal V$, that is, the cardinality of $Q$. 

The fact that $\lambda(\theta)$ is a real analytic function of $\theta$ is a direct consequence of the following classical phenomenon from complex analysis, when applied to the characteristic polynomial of $L_\theta$.

\begin{Lem}\label{root} Let $d\geq 1$ and $\mathbb C_d[z]$ be the set complex polynomial with degree $\leq d$ and $P$ be in $\mathbb C_d[z]$. Assume $0\leq k\leq d$ and $U$ is an open subset of $\mathbb C$ such that $P$ has $k$ roots (counted with multiplicities) in $U$. Then, any polynomial $Q$ which is close enough to $P$ in $\mathbb C_d[z]$ has exactly $k$ roots (counted with multiplicities) in $U$. If $k=1$, the unique root of $Q$ in $U$ is a holomorphic function of $Q$ in a neighborhood of $P$.
\end{Lem}

We will actually need to know a bit more, namely that the projections on each summand of the direct sum 
$\mathcal V=\mathbb R f_\theta\oplus (L_\theta-\lambda(\theta))\mathcal V$ depend analytically of $\theta$. This follows from the rather classical

\begin{Lem}\label{residue} Let $H$ be a finite-dimensional complex vector space, $T$ be an endomorphism of $H$ and $\lambda$ be an eigenvalue of $T$. Let $L\subset H$ be the associated characteristic subspace and $M$ be the unique $T$-invariant complementary subspace of $L$, that is, we have $L=\ker(T-\lambda)^d$ and $M=(T-\lambda)^d H$ where $d$ is the dimension of $H$. Then, the residue at $\lambda$ of the endomorphism valued meromorphic function on $\mathbb C$,
$$z\mapsto(z-T)^{-1}$$
is the projection on $L$ in the direct sum $H=L\oplus M$.
\end{Lem}

\begin{proof} As $T$ preserves both $L$ and $M$, it suffices to deal with the case where $H=L$. Then, we need to show that the residue 
of $(z-T)^{-1}$ at $\lambda$ is the identity map. Indeed, since we have assumed that $\lambda$ is the unique eigenvalue of $T$, we may write $T$ as $T=\lambda+N$ where $N$ is a nilpotent operator. Then, for $z\neq 0$, we have 
$$(z-N)^{-1}=\sum_{k=0}^\infty z^{-k-1}N^k$$
(the sum being actually finite).
This gives, for $z\neq \lambda$,
$$(z-T)^{-1}=\sum_{k=0}^\infty (z-\lambda)^{-k-1}N^k.$$
The conclusion follows.
\end{proof}

\begin{Cor}\label{analytic} Let $\mathcal G=(Q,A,\sigma,\gamma)$ be a connected directed graph.
Then, the maps $\theta\mapsto\lambda(\theta)$,
$\theta\mapsto f_\theta$ and $\theta\mapsto \varphi_\theta$ are real analytic on $\mathcal E$.
\end{Cor}

\begin{proof} We will apply Lemma \ref{residue} in the complexification $\mathcal V_\mathbb C$ of $\mathcal V$, which may be considered as the space of complex valued functions on $Q$. For $\theta$ in $\mathcal E_\mathbb C$, we still define $L_\theta$ by the same formula as in the real case, that is, for $f$ in $\mathcal V_\mathbb C$ and $q$ in $Q$,
$$L_\theta(f)(q)=\sum_{\substack{a\in A\\ \sigma(a)=q}}e^{-\theta(a)}f(\gamma(a)).$$
The map $\theta\mapsto L_\theta$ is holomorphic on $\mathcal E_\mathbb C$. For $\theta$ a real function, we now that the eigenvalue $\lambda(\theta)$ has multiplicity $1$. Therefore, by Lemma \ref{root}, for $\xi$ in $\mathcal E_\mathbb C$ close to $\theta$, the operator $L_\xi$ has a unique eigenvalue $\lambda(\xi)$ which is close to $\lambda(\theta)$, this eigenvalue has multiplicity $1$ and is a holomorphic function of $\xi$. 
Thus, the function $\theta\mapsto\lambda(\theta)$ is well defined and holomorpic in a neighbourhood $U$ of $\mathcal E$ in $\mathcal E_\mathbb C$.

For $\theta$ in $U$, we let $\Pi_\theta$ be the projection on $\ker(L_\theta-\lambda(\theta))$ in the direct sum 
$\mathcal V_\mathbb C=\ker(L_\theta-\lambda(\theta))\oplus (L_\theta-\lambda(\theta))\mathcal V_\mathbb C$. Note that by construction, $\Pi_\theta$ has rank one.
Besides, let us fix $\theta$ in $U$. Choose a closed disk $D$ in $\mathbb C$ whose interior contains $\lambda(\theta)$ but such that $D$ contains no other value of $L_\theta$. Then, for $\xi$ close enough to $\theta$, we have, by Lemma \ref{residue},
$$\Pi_\xi=\frac{1}{2i\pi}\int_{\partial D}(z-L_\xi)^{-1}\de z,$$
hence $\Pi_\xi$ is a holomorphic function of $\xi$.
By construction, for $\theta$ in $\mathcal E$, we have 
$$\Pi_\theta({\bf 1})=\frac{\langle \varphi_\theta ,{\bf 1}\rangle}{\langle \varphi_\theta ,f_\theta\rangle}f_\theta=f_\theta,$$
hence $f_\theta$ is an analytic function of $\theta$.

By reasoning in the same way for the adjoint operator of $L_\theta$, we obtain that $\varphi_\theta$ is also an analytic function of $\theta$.
\end{proof} 

In order to describe the second derivative of $\lambda$ in Proposition \ref{lambdaderivative}, we will use certain
complementary subspaces of $\nabla \mathcal V\oplus \mathbb R{\bf 1}$ in $\mathcal E$.
Given $\theta$ in $\mathcal E$, we say that an element $\varphi$ of $\mathcal E$ is $\theta$-balanced if, for every $q$ in $Q$, we have
$$\sum_{\substack{a\in A\\ \sigma(a)=q}}e^{-\theta(a)}\varphi(a)f_\theta(\gamma(a))=0.$$
In other words, $\varphi$ is $\theta$-balanced if we have $\de_\theta L(\varphi)f_\theta=0$ where $\de_\theta L$ is the differential at $\theta$ of the map $L:\xi\mapsto L_\xi$.
The space of $\theta$-balanced functions in $\mathcal E$ is denoted by $\mathcal E^\theta$.

\begin{Lem}\label{injective2} Let $\mathcal G=(Q,A,\sigma,\gamma)$ be a connected directed graph. Fix $\theta$ in $\mathcal E$. Then, we have 
$$\mathcal E=\mathbb R{\bf 1}\oplus \nabla \mathcal V\oplus \mathcal E^{\theta},$$
that is, every $\varphi$ in $\mathcal E$ may be written as $\varphi=c+\nabla g+\psi$ where $c$ is a real number, $g$ is in $\mathcal V$ and $\psi$ is $\theta$-balanced. The constant $c$ and the function $\psi$ are uniquely defined; the function $g$ is defined up to the addition of a constant function.
\end{Lem}

\begin{proof} Fix $\varphi$ in $\mathcal E$ and define a function $f$ on $Q$ by setting, for $q$ in $Q$,
$$f(q)=\sum_{\substack{a\in A\\ \sigma(a)=q}}e^{-\theta(a)}\varphi(a)f_\theta(\gamma(a))=-d_\theta L(\varphi)f_\theta(q).$$
Then, since $\lambda(\theta)$ is a simple eigenvalue of $L_\theta$ and $f_\theta$ is a positive function, we may write
\begin{equation}\label{eqbalanced}f=c\lambda f_\theta+L_\theta (gf_\theta)-\lambda(\theta)gf_\theta\end{equation}
for some $c$ in $\mathbb R$ and $g$ in $\mathcal V$. The number $c$ is uniquely determined by this equation; the function $g$ is unique up to a constant. We set $\psi=\varphi-c-\nabla g$ and \eqref{eqbalanced} directly says that $\psi$ is $\theta$-balanced. Uniqueness follows 
the uniqueness properties of the solutions of \eqref{eqbalanced}.
\end{proof}

We can now compute the first derivatives of $\lambda$. This is inspired by standard computations in the theory of finite state space Markov chains.

\begin{proof}[Proof of Proposition \ref{lambdaderivative}]
We fix $\theta$ in $\mathcal E$ and we will study the function $\lambda$ in the neighbourhood of $\theta$. We change the normalization of the eigenvector, so that we set, for $\xi$ in $\mathcal E$, $g_\xi=\frac{1}{\langle \varphi_\theta ,f_\xi\rangle}f_\xi$ in order to ensure that 
\begin{equation}\label{derivative1}\langle \varphi_\theta ,g_\xi\rangle=1.\end{equation} 
Note that $g_\xi$ is an analytic function of $\xi$ by Corollary \ref{analytic}.

We introduce some notation.
For $\xi$ in $\mathcal E$, we set $h_\xi=f_\theta^{-1}\de_\theta g(\xi)$ where $ \de_\theta g$ is the differential at $\theta$ of the map $\xi\mapsto g_\xi$. Note that $h_\xi$ is a function on $Q$ which depends linearly on $\xi$. 
By derivating \eqref{derivative1}, we obtain 
\begin{equation}\label{derivative2}\langle\varphi_\theta,h_\xi f_\theta\rangle=0.\end{equation}

We study the first derivative of $\lambda$. For $\xi$ in $\mathcal E$, we have
\begin{equation}\label{derivative5}L_\xi g_\xi=\lambda(\xi)g_\xi.\end{equation}
Derivating this equation at $\theta$ yields
\begin{equation}\label{derivative3}\de_\theta L(\xi) f_\theta+L_\theta (h_\xi f_\theta)
=\de_\theta\lambda(\xi)f_\theta+\lambda(\theta)h_\xi f_\theta.\end{equation}
We evaluate $\varphi_\theta$ on the latter. From \eqref{derivative2} and using that 
$L_\theta^*\varphi_\theta=\lambda(\theta)\varphi_\theta$, we get 
$\de_\theta\lambda(\xi)=\langle \varphi_\theta,\de_\theta L(\xi) f_\theta\rangle$. Using the definition of the operator $L_\theta$ yields
\begin{multline*}
\de_\theta\lambda(\xi)=-\sum_{q\in Q}\varphi_\theta(q)\sum_{\substack{a\in A\\ \sigma(a)=q}}e^{-\theta(a)}\xi(a)f_\theta(\gamma(a))
\\=-\sum_{a\in A}\varphi_\theta(\sigma(a))e^{-\theta(a)}\xi(a)f_\theta(\gamma(a))\end{multline*}
as required.

Now, we prove \eqref{invariance}. For $c$ in $\mathbb R$, we have $L_{\theta+c}=e^{-c}L_\theta$, hence $\lambda(\theta+c)=e^{-c}\lambda(\theta)$. Fix $\varphi$ in $\mathcal V$. For $f$ in $\mathcal V$ and $q$ in $Q$, we have
$$L_{\theta+\nabla\varphi}(f)=e^{\varphi(q)}\sum_{\substack{a\in A\\ \sigma(a)=q}}e^{-\theta(\gamma(a))-\varphi(\gamma(a))}f(\gamma(a))
=e^{\varphi(q)}L_\theta(e^{-\varphi}f)(q).$$
Thus, the operators $L_{\theta+\nabla\varphi}$ and $L_\theta$ are conjugated. Therefore, they have the same spectral radius, that is, we have $\lambda(\theta+\nabla\varphi)=\lambda(\theta)$.

Finally, we study the second derivative of $\lambda$. We notice that, in view of \eqref{invariance}, for $\xi$ in $\mathcal E$ and $\eta$ in 
$\nabla \mathcal V\oplus\mathbb R{\bf 1}$, we have $\de^2_\theta \log\lambda(\xi,\eta)=0$. In other words, the space $\nabla \mathcal V\oplus\mathbb R{\bf 1}$ is contained in the kernel of the symmetric bilinear form $\de^2_\theta \log\lambda(\xi,\eta)$. 
Therefore, by Lemma \ref{injective2}, it suffices to compute this bilinear form on the space $\mathcal E^\theta$ of $\theta$-balanced elements. Note also that, in view of \eqref{derivative6}, if $\xi$ and $\eta$ are $\theta$-balanced, we have 
$\de_\theta\lambda(\xi)=\de_\theta\lambda(\eta)=0$, hence
\begin{equation}\label{derivative7}\de^2_\theta\log\lambda(\xi,\eta)=\lambda(\theta)^{-1}\de^2\lambda(\xi,\eta).\end{equation}

Notice also that
\eqref{derivative3} says that, for $\xi$ in $\mathcal E$, the function 
$$\xi+\lambda(\theta)^{-1}\de_\theta\lambda(\xi)-\nabla h_\xi$$
is $\theta$-balanced. In particular, if $\xi$ is already $\theta$-balanced, from \ref{injective2}, 
we get that $\de_\theta\lambda(\xi)=0$ (which we already knew) and $h_\xi$ is a constant function on $Q$. In view of \eqref{derivative2}, this implies $h_\xi=0$.

Now, we derivate \eqref{derivative5} twice and we get, for $\xi,\eta$ in $\mathcal E$,
\begin{multline*}\de^2_\theta L(\xi,\eta) f_\theta+ \de_\theta L(\xi) (h_\eta f_\theta)+\de_\theta L(\eta) (h_\xi f_\theta)
+L_\theta (\de^2_\theta g(\xi,\eta))\\=\de^2_\theta\lambda(\xi,\eta)f_\theta+
\de_\theta \lambda(\xi) h_\eta f_\theta+\de_\theta \lambda(\eta) h_\xi f_\theta
+\lambda(\theta)\de^2_\theta g(\xi,\eta).\end{multline*}
Assume $\xi$ and $\eta$ are $\theta$-balanced. Then, we get
$$\de^2_\theta L(\xi,\eta) f_\theta
+L_\theta (\de^2_\theta g(\xi,\eta))=\de^2_\theta\lambda(\xi,\eta)f_\theta+
\lambda(\theta)\de^2_\theta g(\xi,\eta).$$
From \eqref{derivative1}, we get $\langle\varphi_\theta,\de^2_\theta g(\xi,\eta)\rangle=0$. Thus, applying $\varphi_\theta$ to the above 
and using the definition of $L_\theta$ yields, still when $\xi$ and $\eta$ are $\theta$-balanced,
$$\de^2_\theta\lambda(\xi,\eta)=\langle \varphi_\theta,\de^2_\theta L(\xi,\eta) f_\theta\rangle
=\sum_{a\in A}\varphi_\theta(a)e^{-\theta(a)}\xi(a)\eta(a)f_\theta(a).$$
Thus, by \eqref{derivative7}, the symmetric bilinear form $\de^2_\theta\log\lambda$ is positive definite on $\mathcal E^\theta$. The conclusion follows.
\end{proof}

\section{Regularity of $\psi_{\mathcal G}$}

We can use the constructions above to say a bit more about the function $\psi_{\mathcal G}$.

\begin{Cor} \label{psiregular} Let $\mathcal G=(Q,A,\sigma,\gamma)$ be a connected directed graph. 
The asymptotic cone $\mathcal C_{\mathcal G}$ of the set $\mathscr P(W)$ is also the essential definition cone of the function $\psi_{\mathcal G}$ in $\mathcal E$. This cone $\mathcal C_{\mathcal G}$ is contained in the orthogonal subspace $(\nabla \mathcal V)^\perp$ of $\nabla \mathcal V$ and $\mathcal C_{\mathcal G}$ has non empty interior in 
$(\nabla \mathcal V)^\perp$. The function $\psi_{\mathcal G}$ is analytic, positive and strictly concave in the interior of $\mathcal C_{\mathcal G}$ in 
$(\nabla \mathcal V)^\perp$.
\end{Cor}

Recall that the asymptotic cone of a subset $X$ of $\mathcal E$ is the set of limit points of sequences of the form $(t_nx_n)_{n\geq 0}$ where $(x_n)_{n\geq 0}$ is a sequence in $X$ and $(t_n)_{n\geq 0}$ is a sequence of non-negative real numbers converging to $0$.

\begin{proof} By Proposition \ref{lambdaderivative}, the function $\lambda$ is invariant under translations by vectors from $\nabla \mathcal V$. By Proposition
\ref{omegatransfer}, the set $\Omega_{\mathcal G}$ is also invariant by these translations. Therefore, by Lemma \ref{translation}, 
the essential definition cone $\mathcal C_{\mathcal G}$ of $\psi_{\mathcal G}$ is contained in $(\nabla \mathcal V)^\perp$. 
By Lemma \ref{psicounting}, $\mathcal C_{\mathcal G}$ is also the asymptotic cone of $\mathscr P(W)$ and $\psi_{\mathcal G}$ is $\geq 0$ on $\mathcal C_{\mathcal G}$.

Note that the fact that the asymptotic cone of $\mathscr P(W)$ is contained in $(\nabla \mathcal V)^\perp$ can be obtained directly: indeed, for $f$ in $\mathcal V$ and $w$ in $W$, a word with length $n$, one has 
$$\langle \mathscr P(w),\nabla f\rangle=\sum_{k=1}^nf(\gamma (w_k))-f(\sigma(w_k))=f(\gamma(w_n))-f(\sigma(w_1)),$$
where we have used that, for $2\leq k\leq n$, we have $\sigma(w_k)=\gamma(w_{k-1})$.
Therefore, the vectors $\mathscr P(w)$ are at a bounded distance from the space $(\nabla \mathcal V)^\perp$.

Finally, we note that since by Proposition \ref{lambdaderivative}, the function $\lambda$ is invariant under translations by vectors from $\nabla \mathcal V$, we can consider it as a function on the quotient space $\mathcal E/\nabla \mathcal V$. Besides, the derivative of $\lambda$ never vanishes on constant functions. Then, by Proposition \ref{lambdaderivative},
the assumption of Lemma \ref{local} is satisfied and the analycity properties of $\psi_{\mathcal G}$ follow.
\end{proof}

This leads to the introduction of a new object asocciated with the language $W$.

\begin{Cor} Let $\mathcal G=(Q,A,\sigma,\gamma)$ be a connected directed graph. 
Then, there exists a unique vector $x_{\mathcal G}$ in $\mathcal E$ with
$\langle{\bf 1},x_{\mathcal G}\rangle=1$ and 
$$\psi_{\mathcal G}(x_{\mathcal G})=\sup_{\substack{x\in \mathcal E\\ \langle{\bf 1},x\rangle=1}}\psi_{\mathcal G}(x).$$
One has $\psi_{\mathcal G}(x_{\mathcal G})=\delta_{\mathcal G}$ and $x_{\mathcal G}$ belongs to $(\nabla \mathcal V)^\perp$.
\end{Cor}

\begin{proof} This directly follows from the strict concavity property of $\psi_{\mathcal G}$ established in  Corollary \ref{psiregular}.\end{proof}

\begin{Ex} \label{fibonacci4} We go back to the Fibonacci graph of Examples \ref{fibonacci1}, \ref{fibonacci2} and \ref{fibonacci3} and we will compute the vector $x_\mathcal G$. We look for vectors $x=(x_1,x_2,x_3)$ in $(\nabla \mathcal V)^\perp$ with $x_1+x_2+x_3=1$, that is, vectors of the form $x=(1-2u,u,u)$. For such a vector the function $\psi_\mathcal G(1-2u,u,u)$ is infinite if $u\notin [0,\frac{1}{2}]$ and the value of the interval is given by
$$\rho(u)=\psi_\mathcal G(1-2u,u,u)=(1-u)\log(1-u)-(1-2u)\log(1-2u)-u\log u.$$
The derivative is
$$\rho'(u)=-\log(1-u)+2\log(1-2u)-\log(u)=\log\frac{(1-2u)^2}{u(1-u)},$$
which vanishes at a solution of the equation
$(1-2u)^2=u(1-u)$. We rewrite the equation as
$5u^2-5u+1=0$. The only solution of this equation in $[0,\frac{1}{2}]$
is $\frac{1}{2}(1-\frac{1}{\sqrt{5}})$. We get
$$x_\mathcal G=\left(\frac{1}{\sqrt{5}},\frac{1}{2}\left(1-\frac{1}{\sqrt{5}}\right),\frac{1}{2}\left(1-\frac{1}{\sqrt{5}}\right)\right).$$
\end{Ex}

\section{A Central Limit Theorem}
\label{secCLT}
We now aim at counting the number of elements of $w$ such that $\mathscr P(w)$ belongs to certain subsets of $\mathcal E$. This will rely on developing a formalism which is an adaptation of the limit theorems of probability theory for finite state space Markov chains. As a first step, we begin by writing a version of the Central Limit Theorem.

For $\theta$ in $\mathcal E$, we set $x_\theta=-\de_\theta\log\lambda$. Note that \eqref{invariance} implies that $\langle {\bf 1},x_\theta\rangle=1$.
We use the notation introduced in Section \ref{secgraphs} and Section \ref{secglobalcounting}.

\begin{Thm} \label{CLT} Let $\mathcal G=(Q,A,\sigma,\gamma)$ be a connected directed graph. 
Choose $\theta$ in $\mathcal E$ with $\lambda(\theta)=1$. 
Fix $q$ in $Q$. Then, we have
$$\frac{1}{n}\sum_{w\in W_n^q}e^{-\langle\theta,\mathscr P(w)\rangle}(\mathscr P(w)-nx_\theta) \td{n}{\infty}0$$
and 
there exists a positive definite quadratic form $\chi_\theta$ on the space $\mathcal E_0=(\mathbb R{\bf 1}+\nabla \mathcal V)^\perp$ 
such that, for any continuous compactly supported function $g$ on $\mathcal E$, we have
\begin{multline*}\sum_{w\in W_n^q}e^{-\langle\theta,\mathscr P(w)\rangle}g\left(\frac{1}{\sqrt{n}}(\mathscr P(w)-nx_\theta)\right)\\
-\frac{\langle \varphi_\theta,{\bf 1}_{Q^q_n}\rangle f_\theta(q)}{(2\pi)^{r}}
\int_{\mathcal E_0} \exp\left(-\frac{1}{2}\chi_\theta(y)\right)g(y)\de y
\td{n}{\infty}0.
\end{multline*}
\end{Thm}

We have denoted by $r$ the dimension of $\mathcal E_0$. We have equipped the space $\mathcal E_0$ with the Lebesgue measure associated with $\chi_\theta$, that is, the unit hypercubes with respect to $\chi_\theta$ have measure $1$. See \cite{BCGL} for a related result.

The first statement of the Proposition says that, when the set $W_n^q$ is equipped with the measure associated to the weights 
$e^{-\langle\theta,\mathscr P(w)\rangle}$, $w\in W_n^q$, the random vector $\mathscr P(w)$ behaves as a random walk on $\mathcal E$ which satisfies a Law of Large Numbers with expectation $x_\theta$. The second statement says that, when taking out this drift, this random walk satisfies a Central Limit Theorem, that is, after being normalized in the right way, it admits an asymptotic Gaussian distribution supported on $(\nabla \mathcal V)^\perp\subset \mathcal E$.

To prove this second statement, we shall use the so called method of characteristic functions from probability theory, which is to say the following 

\begin{Lem} \label{characteristic}
Let $H$ be a real finite-dimensional vector space and $(\mu_n)_{n\geq 0}$ and $\mu$ be finite Borel measures on $H$. Then, the following are equivalent\\
{\em (i)} One has $\mu_n(V)\td{n}{\infty}\mu(V)$ and, for any continuous compactly supported function $g$ on $H$, 
$\int_Hg\de\mu_n\td{n}{\infty}\int_H g\de\mu$.\\
{\em (ii)} For any linear functional $\varphi$ on $H$, one has
$$\int_H\exp(i\varphi(x))\de\mu_n(x)\td{n}{\infty}\int_H \exp(i\varphi(x))\de\mu(x).$$
\end{Lem}

\begin{proof}[Proof of Proposition \ref{CLT}] We first establish the first statement. The main idea is that the quantity 
$G(n,q,\theta)=-\sum_{w\in W_n^q}e^{-\langle\theta,\mathscr P(w)\rangle}\mathscr P(w)$ may be interpreted 
as the derivative at $\theta$ of $L_\theta^n{\bf 1}(q)$. We then use the precise structure of the operator $L_\theta$ to conclude.

More precisely, by Lemma \ref{modulus1} and Corollary \ref{asymptotic1} we may write, for $\xi$ in $\mathcal E^*$ and $n\geq 0$,
\begin{equation}\label{eqLn}L_\xi^n{\bf 1}(q)=
\lambda(\xi)^np\langle \varphi_\xi,{\bf 1}_{Q^q_n}\rangle f_\xi(q)
+\Pi_\xi L_\xi^n {\bf 1}(q),\end{equation}
where $\Pi_\xi$ is a projection of $\mathcal V$ onto a subspace of codimension $p$, which is stable under $L_\xi$ and where the spectral radius of $L_\xi$ is $<\lambda(\xi)$. Using Lemma \ref{root} and Lemma \ref{residue} as in the proof of Proposition \ref{lambdaderivative} shows that $\Pi_\xi$ is an analytic function of $\xi$. As $\lambda(\theta)=1$, derivating \eqref{eqLn} at $\theta$ yields
\begin{multline}\label{eqderivLn}
\xi(G(n,a,\theta))=
pn\de_\theta\lambda(\xi)\langle \varphi_\theta,{\bf 1}_{Q^q_n}\rangle f_\theta(q)+
p\langle \de_\theta\varphi(\xi),{\bf 1}_{Q^q_n}\rangle f_\theta(q)
\\+p\langle \varphi_\theta,{\bf 1}_{Q^q_n}\rangle \de_\theta f(\xi)(q)+
\de_\theta\Pi(\xi) L_\theta^n {\bf 1}(q)
+\Pi_\theta\sum_{k=1}^n L_\theta^{k-1} \de_\theta L(\xi) L_{\theta}^{n-k} {\bf 1}(q).\end{multline} 
The conclusion will follow from the fact that, in the above equation, all terms but the first are bounded and hence go to $0$ when divided by $n$. Indeed, as $\lambda(\theta)=1$, the operator $L_\theta$ has spectral radius $1$ and, by Lemma \ref{modulus1}, all eigenvalues of $L_\theta$ with modulus $1$ are simple. Hence, the operator norm of $L_\theta^n$ is uniformly bounded. Besides, all eigenvalues of $L_\theta$ in the range of $\Pi_\theta$ have modulus $<1$, 
hence, there exists $C>0$ and $0<\alpha<1$ such that, for any $n\geq 0$, the operator $\Pi_\theta L_\theta^n$ has norm $\leq \alpha^n C$. Thus, we get, up to enlarging $C$,
$$\|\de_\theta\Pi(\xi) L_\theta^n {\bf 1}\|\leq C\|{\bf 1}\|\|\xi\|$$
and
$$\|\Pi_\theta\sum_{k=1}^n L_\theta^{k-1} \de_\theta L(\xi) L_{\theta}^{n-k} {\bf 1}\|
\leq C^2\sum_{k=1}^n\alpha^{k-1}\|{\bf 1}\|\|\xi\|\leq \frac{C^2}{1-\alpha}\|{\bf 1}\|\|\xi\|.$$
Plugging this in \eqref{eqderivLn} yields
$$\frac{1}{n}G(n,q,\theta)+\langle \varphi_\theta,{\bf 1}_{Q^q_n}\rangle f_\theta(q)x_\theta\td{n}{\infty}0.$$
Besides, by Lemma \ref{modulus1} and Corollary \ref{asymptotic1} we have 
$$\sum_{w\in W_n^q}e^{-\langle\theta,\mathscr P(w)\rangle}-p\langle \varphi_\theta,{\bf 1}_{Q^q_n}\rangle f_\theta(q)
=L_\theta^n{\bf 1}(q)-p\langle \varphi_\theta,{\bf 1}_{Q^q_n}\rangle f_\theta(q)\td{n}{\infty}0.$$
The conclusion follows.

Now, we will prove the second statement. The proof is the same as the standard proof of the Central Limit Theorem, which relies on 
Lemma \ref{characteristic}. Thus, for $\xi$ in $\mathcal E^*$ and $n\geq 0$, we compute
\begin{multline*}
\sum_{w\in W_n^q}\exp\left(-\langle\theta,\mathscr P(w)\rangle+i\xi\left(\frac{1}{\sqrt{n}}(\mathscr P(w)-nx_\theta)\right)\right)\\
=\exp(-i\sqrt{n}\xi(x_\theta))
L_{\theta-i\frac{\xi}{\sqrt{n}}}^n {\bf 1}(q).\end{multline*}
Again, we will use the structure of the operator $L_{\theta-i\frac{\xi}{\sqrt{n}}}$ to conclude.
When $n$ is large, $\theta-i\frac{\xi}{\sqrt{n}}$ is close to $\theta$, hence we can write as in \eqref{eqLn},
\begin{multline*}L_{\theta-i\frac{\xi}{\sqrt{n}}}^n {\bf 1}(q)
=\\ \lambda\left(\theta-i\frac{\xi}{\sqrt{n}}\right)^n
p\langle \varphi_{\theta-i\frac{\xi}{\sqrt{n}}},{\bf 1}_{Q^q_n}\rangle f_{\theta-i\frac{\xi}{\sqrt{n}}}(q)
+\Pi_{\theta-i\frac{\xi}{\sqrt{n}}} L_{\theta-i\frac{\xi}{\sqrt{n}}}^n {\bf 1}(q).\end{multline*}
As above, the spectral radius of $L_\theta$ in the range of $\Pi_\theta$ is $<1$ and we get
$$ \Pi_{\theta-i\frac{\xi}{\sqrt{n}}} L_{\theta-i\frac{\xi}{\sqrt{n}}}^n \td{n}{\infty}0$$
and, by continuity,
$$\langle \varphi_{\theta-i\frac{\xi}{\sqrt{n}}},{\bf 1}_{Q^q_n}\rangle f_{\theta-i\frac{\xi}{\sqrt{n}}}
\td{n}{\infty}\langle \varphi_{\theta},{\bf 1}_{Q^q_n}\rangle f_{\theta}.$$
The study of the remaining term is the same as in the classical proof of the Central Limit Theorem. Indeed, 
from the computation of the derivatives of $\lambda$ in Proposition \ref{lambdaderivative}, we can write, for $n$ large,
$$\exp\left(-i\frac{1}{\sqrt{n}}\xi(x_\theta)\right)\lambda\left(\theta-i\frac{\xi}{\sqrt{n}}\right)
=\exp\left(-\frac{1}{2n}\de_\theta^2\log\lambda(\xi,\xi)+\frac{a_n}{n^{\frac{3}{2}}}\right),
$$
where $(a_n)_{n\geq 0}$ is a bounded sequence of complex numbers.
We obtain
\begin{multline*}\sum_{w\in W_n^q}\exp\left(-\langle\theta,\mathscr P(w)\rangle+i\xi\left(\frac{1}{\sqrt{n}}(\mathscr P(w)-nx_\theta)\right)\right)
\\ -\exp\left(-\frac{1}{2}\de_\theta^2\log\lambda(\xi,\xi)\right)
p\langle \varphi_{\theta},{\bf 1}_{Q^q_n}\rangle f_{\theta}(q)
\td{n}{\infty}0\end{multline*}
and the result is a consequence of the description of the null space of $\de_\theta^2\log\lambda$ in Proposition \ref{lambdaderivative},
Lemma \ref{characteristic} and standard computations of Fourier transforms of Gaussian measures.
\end{proof}

\section{Local asymptotic estimates for languages of finite type}\label{secLLT}
We continue to translate results from the theory of finite state space Markov chains into our language. We will now show a strong version of the Local Limit Theorem in the spirit of \cite{BoMo}.

More precisely, for $q,q'$ in $Q$ and $n\geq 0$, we will consider the set $W_n^{q,q'}$ and give an estimate as $n\rightarrow\infty$ 
of the number of words $w$ in $W_n^{q,q'}$ such that $\mathscr P(w)$ takes a fixed value $x$ in $\mathcal E$. 
Roughly speaking, this estimate will say that this number behaves as $C_x n^{-\frac{r}{2}}e^{-\psi_{\mathcal G}(x)}$, where $r$ is an integer and $C_x>0$ only depends on the direction of the vector $x$ (that is, $C_{tx}=C_x$ for $t>0$).
This can be seen as a strong large deviation principle in the sense of \cite{DZ}. In stating such estimates precisely, we encounter several difficulties. 

First, if we want the above mentioned number to be non zero, we have to assume that $\langle {\bf 1},x\rangle=n$, since $\langle {\bf 1},\mathscr P(w)\rangle=n$ for every $w$ in $W^{q,q'}_n$. In other words, we will assume that $x$ belongs to a fixed affine hyperplane of the space $\mathcal E$. 

In the same spirit, for $w$ in $W$, the vector $\mathscr P(w)$ is integer valued when seen as a function on $A$. Therefore, we will need to assume that $x$ belongs to the lattice $\Lambda\subset \mathcal E$ of functions on $A$ with values in $\mathbb Z$.

Finally, as we have seen in the proof of Corollary \ref{psiregular}, the set $\mathscr P(W)$ stays at a bounded distance from the vector space $(\nabla \mathcal V)^\perp$. Nevertheless, in general $\mathscr P(W)$ is not contained in $(\nabla \mathcal V)^\perp$. The next Lemma will show us that the elements of the set $\mathscr P(W^{q,q'})$ all belong to a translate of $(\nabla \mathcal V)^\perp$ that only depends on $q$ and $q'$.

\begin{Lem} \label{straight} Let $\mathcal G=(Q,A,\sigma,\gamma)$ be a connected directed graph. 
There exists a map $\mathscr R:Q\rightarrow \Lambda$ with the following properties. \\
{\em (i)} For any $a$ in $A$, we have 
$${\bf 1}_{a}-\mathscr R(\gamma(a))+\mathscr R(\sigma(a))\in (\nabla \mathcal V)^\perp.$$\\
{\em (ii)} For any $q$ in $Q$ and $q'$ in $Q^q_0$ we have
$\langle {\bf 1},\mathscr R(q)-\mathscr R(q')\rangle=0$.\\
In particular, for $n\geq 0$, $q,q'$ in $Q$ and $w$ in $W_{pn}^{q,q'}$, we have 
$$\mathscr P(w)-\mathscr R(q')+\mathscr R(q)\in (\nabla \mathcal V)^\perp\cap\Lambda \mbox{ and } \langle {\bf 1},\mathscr P(w)-\mathscr R(q')+\mathscr R(q)\rangle=n.$$
\end{Lem}

The sets $Q_j^q$, $q\in Q$, $j\in\mathbb Z/p\mathbb Z$, were introduced in Section \ref{secgraphs}. We have denoted by ${\bf 1}_a$ the indicator function of the set $\{a\}\subset A$.

\begin{Ex} \label{fibonacci5}
In the Fibonacci graph of Examples \ref{fibonacci1}, \ref{fibonacci2}, \ref{fibonacci3} and \ref{fibonacci4}, 
we can take $\mathscr R(1)=(0,0,0)$ and $\mathscr R(2)=(1,0,-1)$.
\end{Ex}

We denote by $\Delta\subset \mathcal V$ the lattice of functions on $Q$ with values in $\mathbb Z$. 

\begin{proof}[Proof of Lemma \ref{straight}] 
The construction relies on elementary properties of linear algebra. We fix $q_0$ in $Q$.

First, we build a section $S:\nabla \mathcal V\rightarrow \mathcal V$ of the linear map $\nabla$. By Lemma \ref{injective}, for $f$ in $\nabla \mathcal V$, there exists a unique $g$ in $\mathcal V$ with $g(q_0)=0$ and $\nabla g=f$. We set $Sf=g$. We claim that if $f$ belongs to $\Lambda$, then $g=Sf$ belongs to $\Delta$. Indeed, for such a function $f$, let $Q'\subset Q$ be the set of those $q$ in $Q$ such that $g(q)$ is an integer. We have $q_0\in Q$. Besides, for $q$ in $Q'$ and $a$ in $A$ wth $\sigma(a)=q$, by construction, we have 
$f(a)=g(\gamma(a))-g(q)$ hence $g(\gamma(a))=f(a)+g(q)$ is also an integer. As $\mathcal G$ is connected, we obtain $Q'=Q$.

We have just shown that 
\begin{equation}\label{grpdecompo1}
S(\Lambda\cap \nabla \mathcal V)\subset \Delta.\end{equation}
Note that $\Lambda\cap \nabla \mathcal V=S(\Delta)$ is a lattice in $\nabla \mathcal V$ and that $\Lambda\cap \mathbb R{\bf 1}$ is a lattice in (the line!) $\mathbb R{\bf 1}$. Therefore, $\Lambda\cap (\nabla \mathcal V\oplus\mathbb R{\bf 1})$ is a lattice in $\nabla \mathcal V\oplus\mathbb R{\bf 1}$.
We will extend the definition of $S$ in order to get a linear map 
$\Lambda\cap (\nabla \mathcal V\oplus \mathbb R{\bf 1})\rightarrow \Delta$. This requires us to find a complementary subgroup of 
$\Lambda\cap \nabla \mathcal V$ in $\Lambda\cap (\nabla \mathcal V\oplus \mathbb R{\bf 1})$. The situation is a bit tricky when $p\geq 2$. 

Recall that by definition, the group $\Lambda\cap (\nabla \mathcal V\oplus \mathbb R{\bf 1})$ is the set of integer valued functions $f$ on $A$
which are of the form $f=\nabla g+c{\bf 1}$ for some $g$ in $\mathcal V$ and $c$ in $\mathbb R$. The difficulty will come from he fact that, when $p\geq 2$, this does not imply that we may choose $c$ to be an integer and $g$ to be integer valued, that is, to belong to $\Delta$. Nevertheless, 
we claim that for such a function, the number $c$ must belong to $p^{-1}\mathbb Z$. Indeed, let $n\geq 0$ be such that $W_n^{q_0,q_0}\neq \emptyset$, and choose a word $w=(a_1,\ldots,a_n)$ in $W_n^{q_0,q_0}$. Then, for any $1\leq i\leq n$, we get
$f(a_i)=g(\gamma(a_i))-g(\sigma(a_i))+c$, hence
$$\sum_{i=0}^{n-1}f(a_i)=g(\gamma(a_n))-g(\sigma (a_0))+nc=nc,$$
so that, by assumption, $nc$ is an integer. By Lemma \ref{modulus12}, we obtain $pc\in\mathbb Z$ as required.

Conversely, we set $g_0=p^{-1}\sum_{j=0}^{p-1}j{\bf 1}_{Q^{q_0}_j}$ and we obtain that $f_0=\nabla g_0-p^{-1}\bf 1$ is integer valued.
Thus, we have the decomposition
\begin{equation}
\label{grpdecompo2}\Lambda\cap (\nabla \mathcal V\oplus \mathbb R{\bf 1})=\mathbb Zf_0\oplus (\Lambda\cap \nabla \mathcal V).\end{equation}
We extend the linear map $S:\nabla \mathcal V\rightarrow \mathcal V$ into a linear map 
$\nabla \mathcal V\oplus \mathbb R{\bf 1}\rightarrow \mathcal V$ by setting $S(f_0)=0$.
From \eqref{grpdecompo1} and \eqref{grpdecompo2}, we obtain
$$S(\Lambda\cap (\nabla \mathcal V\oplus \mathbb R{\bf 1}))\subset \Delta.$$
Besides, since ${\bf 1}=\nabla(pg_0)-pf_0$ and $g(q_0)=0$, we have
\begin{equation}\label{sigma1}S({\bf 1})=pg_0=\sum_{j=0}^{p-1}j{\bf 1}_{Q^{q_0}_j}.\end{equation}

Finally, since $\Lambda\cap (\nabla \mathcal V\oplus \mathbb R{\bf 1})$ is a lattice in $\nabla \mathcal V\oplus \mathbb R{\bf 1}$, we can extend $S$ into a linear map $\mathcal E\rightarrow \mathcal V$ such that $S(\Lambda)\subset \Delta$. 

We will use the adjoint map of $S$ to define $\mathscr R(q)$ for $q$ in $Q$. More precisely, for $q$ in $Q$, we let $\mathscr R(q)$ be the unique element 
of $\mathcal E$ such that, for any $\theta$ in $\mathcal E$, we have 
$$\langle \theta,\mathscr R(q)\rangle=\langle S\theta,{\bf 1}_q\rangle=S\theta(q).$$
Note that, since $S\Lambda\subset\Delta$, for any $\theta$ in $\Lambda$, we have $\langle \theta,\mathscr R(q)\rangle\in\mathbb Z$, hence $\mathscr R(q)$ belongs to $\Lambda$.
Let us check that $\mathscr R$ satisfies the required properties.

Fix $f$ in $\nabla \mathcal V$ and set $g=Sf$, so that, by construction, we have $f=\nabla g$. Then, 
for $a$ in $A$, we get
$$f(a)=g(\gamma(a))-g(\sigma(a))$$
which we rewrite as
$$\langle f,{\bf 1}_a\rangle=\langle Sf,{\bf 1}_{\gamma(a)}-{\bf 1}_{\sigma(a)}\rangle=\langle f,\mathscr R(\gamma(a))-\mathscr R(\sigma(a))\rangle,$$
that is, ${\bf 1}_a-\mathscr R(\gamma(a))+\mathscr R(\sigma(a))\in (\nabla \mathcal V)^\perp$.

Besides, from \eqref{sigma1}, for any $0\leq j\leq p-1$ and 
$q,q'$ in $Q^{q_0}_{j}$, 
$$\langle {\bf 1},\mathscr R(q)\rangle=\langle S{\bf 1},{\bf 1}_q\rangle=S{\bf 1}(q)=j=\langle S{\bf 1},{\bf 1}_{q'}\rangle
=\langle {\bf 1},\mathscr R(q')\rangle,$$
as required.
\end{proof}

We fix a function $\mathscr R$ as in Lemma \ref{straight}.
For $x$ in $\mathcal E$, $n\geq 0$ and $q,q'$ in $Q$, we set 
\begin{equation}\label{localcard}N_n(x,q,q')=|\{w\in W^{q,q'}_{n}|\mathscr P(w)-\mathscr R(q')+\mathscr R(q)=x\}|.\end{equation}
Theorem \ref{LLT} below will tell us that, when $x$ belongs to the intersection 
$\overset{\circ}{\mathcal C}_{\mathcal G}\cap \Lambda$ of the interior of the cone $\mathcal C_{\mathcal G}$ in the space $(\nabla \mathcal V)^\perp$ with the lattice $\Lambda\subset\mathcal E$, 
we can give a very precise estimate of the number $N_n(x,q,q')$.

Let $x$ be in $\mathcal E$ with $\psi_{\mathcal G}(x)>0$. By Corollary \ref{psiregular}, there exists $\theta$ in $\mathcal E$ 
with $\lambda(\theta)=1$ and $\psi_{\mathcal G}(x)=\langle\theta,x\rangle$ and $\theta$ is unique up to translation by an element of $\nabla \mathcal V$. Consider the quadratic form $\chi_\theta$ on the space $\mathcal E_0=(\mathbb R{\bf 1}\oplus \nabla \mathcal V)^\perp$ which is defined in Theorem \ref{CLT}. Note that $\Lambda\cap \mathcal E_0$ is a lattice in $\mathcal E_0$. We set $\varsigma(x)$ to be $|\det \chi_\theta|^{\frac{1}{2}}$ where the determinant is evaluated with respect to a basis of the lattice 
$\Lambda\cap \mathcal E_0$ (and hence does not depend on the choice of this basis).

We use these constructions to state our Local Limit Theorem. We still let $r$ be the dimension of the space $\mathcal E_0=(\mathbb R{\bf 1}\oplus \nabla \mathcal V)^\perp$.

\begin{Thm} \label{LLT}
Let $\mathcal G=(Q,A,\sigma,\gamma)$ be a connected directed graph. 
Fix a compact subset $K$ of 
$\{y\in (\nabla \mathcal V)^\perp|\psi_{\mathcal G}(y)>0\}$.
Then, there exists a sequence $(\varepsilon_n)_{n\geq 0}$ of positive real numbers such that $\varepsilon_n\td{n}{\infty}0$ with the following property. For any $n\geq 0$, $q,q'$ in $Q$ with $q'\in Q_n^q$ and $x$ in $\Lambda\cap (\nabla \mathcal V)^\perp$ with $\langle {\bf 1},x-\mathscr R(q)+\mathscr R(q')\rangle=n$ and 
$n^{-1}x\in K$, we have 
$$\left|(2\pi n)^{\frac{r}{2}}\varsigma(x)e^{-\psi_{\mathcal G}(x)}N_{n}(x,q,q')-
pe^{\langle \theta,\mathscr R(q')-\mathscr R(q)\rangle}\varphi_\theta(q')f_\theta(q)
\right|\leq\varepsilon_n,$$
where $\theta$ is any element of $\mathcal E$ such that $\lambda(\theta)=1$ and $\psi_\mathcal G(x)=\langle \theta,x\rangle$.
\end{Thm}

\begin{Rem} Note that the quantity $e^{\langle \theta,\mathscr R(q')-\mathscr R(q)\rangle}\varphi_\theta(q')f_\theta(q)$ which appears in Theorem \ref{LLT}
does not depend on the choice of $\theta$. Indeed, given a $\theta$ such that $\lambda(\theta)=1$ and $\psi_\mathcal G(x)=\langle \theta,x\rangle$, the other choices are of the form $\theta+\nabla\xi$, where $\xi$ is in $\mathcal V$. Then, since $L_\theta$ and $L_{\theta+\nabla\xi}$ are conjugated by a multiplication operator (see the proof of Proposition \ref{lambdaderivative}), we get, for $q,q'$ in $Q$,
$$\varphi_{\theta+\nabla\xi}(q')f_{\theta+\nabla\xi}(q)=e^{\xi(q)-\xi(q')}\varphi_\theta(q')f_\theta(q).$$ 
On the other hand, by Lemma \ref{straight}, we have
$$\langle \nabla \xi,\mathscr R(q')-\mathscr R(q)\rangle=\xi(q')-\xi(q).$$
Thus, we get
$$e^{\langle \theta+\nabla\xi,\mathscr R(q')-\mathscr R(q)\rangle}\varphi_{\theta+\nabla\xi}(q')f_{\theta+\nabla \xi}(q)
=e^{\langle \theta,\mathscr R(q')-\mathscr R(q)\rangle}\varphi_\theta(q')f_\theta(q).$$
\end{Rem}

The statement of Theorem \ref{LLT} is complicated by the fact that the integer $p$ of Lemma \ref{modulus1} can be $\neq 1$, so that there are obstructions to the numbers $N_{n}(x,q,q')$ being $\neq 0$ coming from the geometry of the graph.
Also, we have stated Theorem \ref{LLT} in order to describe the uniformity properties with respect to $x$
 of the asymptotic estimate of  $N_{n}(x,q,q')$ as $n\rightarrow \infty$. 
When assuming $p=1$, fixing an $x$ in $\Lambda\cap (\nabla \mathcal V)^\perp$ and only caring for what happens on the half line $\mathbb N x$, that is, for the values of $N_{n\langle {\bf 1},x\rangle}(nx,q,q')$, $n\in\mathbb N$, this estimate can be written in the simpler following form. 

\begin{Cor} \label{LLT3}
Let $\mathcal G=(Q,A,\sigma,\gamma)$ be a connected directed graph with $p=1$. 
Fix $q,q'$ in $Q$ and $x$ in $\Lambda\cap (\nabla \mathcal V)^\perp$ with $\psi_{\mathcal G}(x)>0$.
Then, we have
$$N_{n\langle {\bf 1},x\rangle}(nx,q,q')\underset{n\rightarrow\infty}\sim 
\frac{\varphi_\theta(q')f_\theta(q)}{(2\pi \langle {\bf 1},x\rangle)^{\frac{r}{2}}\varsigma(x)}n^{-\frac{r}{2}}e^{n\psi_{\mathcal G}(x)},$$
where $\theta$ is any element of $\mathcal E$ such that $\lambda(\theta)=1$ and $\psi_\mathcal G(x)=\langle \theta,x\rangle$.
\end{Cor}

Compare with \cite[Thm. 5.2]{Melcz}.

We begin the proof of Theorem \ref{LLT} by studying the spectral radius of $L_\theta$ when $\theta$ is not real. In that case, we consider $L_\theta$ as an endomorphism of the complexification $\mathcal V_{\mathbb C}$ of $\mathcal V$, that is, the space of complex valued functions on $Q$.

\begin{Lem}\label{cplxsptctral} Let $\mathcal G=(Q,A,\sigma,\gamma)$ be a connected directed graph. 
Fix $\theta$ in $\mathcal E$. Then, for $\xi$ in $\mathcal E$, the operator $L_{\theta+2i\pi\xi}$ has spectral radius $\leq \lambda(\theta)$ and equality holds if and only if $\xi$ belongs to $\nabla \mathcal V+\mathbb R{\bf 1}+\Lambda$.
\end{Lem}

\begin{proof} First suppose that we may write $\xi\in \nabla\varphi+c{\bf 1}+\Lambda$ where $c$ is in $\mathbb R$ and $\varphi$ is in $\mathcal V$.
Then, as in the proof of Proposition \ref{lambdaderivative}, we get, for $f$ in $\mathcal V_\mathbb C$ and $q$ in $Q$,
$$L_{\theta+2i\pi\xi} f(q)=e^{-2i\pi c}e^{2i\pi\varphi(q)}L_\theta(e^{-2i\pi\varphi} f)(q),$$
hence $L_{\theta+2i\pi\xi}$ is conjugated to $e^{-2i\pi c}L_\theta$ and the spectral radius of $L_{\theta+2i\pi\xi}$ is $\lambda(\theta)$. 

Conversely, first note that, for $f$ in $\mathcal V_\mathbb C$ and $q$ in $Q$, we have 
$$|L_{\theta+2i\pi\xi}f(q)|\leq L_\theta |f|(q),$$ which, by iteration, yields, for $n\geq 0$, 
$|L_{\theta+2i\pi\xi}^nf|\leq L_\theta^n |f|$. Thus, for the operator norm associated with the sup norm, we have 
$\|L_{\theta+2i\pi\xi}^n\|\leq \|L_\theta^n\|$ and $L_{\theta+2i\pi\xi}$ has spectral radius $\leq\lambda(\theta)$.

Suppose now that $L_{\theta+2i\pi\xi}$ admits an eigenvalue of the form $w\lambda(\theta)$ where $w$ is a complex number with modulus $1$. 
Let $\rho\neq 0$ be an eigenvector, that is, we have $L_{\theta+2i\pi\xi}\rho=w\lambda(\theta)\rho$. For $q$ in $Q$, we set 
$\psi(q)=f_\theta(q)^{-1}\rho(q)$, so that we obtain
\begin{equation}\label{eqcplxsptctral1}
\sum_{\substack{a\in A\\ \sigma(a)=q}}\frac{e^{-\theta(a)}f_\theta(\gamma(a))}{\lambda(\theta)f_\theta(q)}e^{-2i\pi\xi(a)}\psi(\gamma(a))
=w\psi(q).
\end{equation}
We denote by $Q'\subset Q$ the set of $q$ in $Q$ such that the modulus $|\psi(q)|$ is maximal. For $q$ in $Q'$, 
in view of \eqref{eqcplxsptctral1}, since 
$$\sum_{\substack{a\in A\\ \sigma(a)=q}}\frac{e^{-\theta(a)}f_\theta(\gamma(a))}{\lambda(\theta)f_\theta(q)}=1$$
and since the disks of $\mathbb C$ are strictly convex, we obtain that, for any $a$ in $A$ with $\sigma(a)=q$, the element $\gamma(a)$ 
also belongs to $Q'$ and satisfies 
\begin{equation}\label{eqcplxsptctral2}\psi(\gamma(a))=we^{2i\pi\xi(a)}\psi(q)=we^{2i\pi\xi(a)}\psi(\sigma(a)).\end{equation}
Thus, as $\mathcal G$ is connected, we get $Q'=Q$ and the function $\psi$ has constant modulus. 
Choose $c$ in $\mathbb R$ with $e^{2i\pi c}=w$ and
a function $\varphi$ on $\mathcal V$ with $e^{2i\pi\varphi}=\psi$. Then, from \eqref{eqcplxsptctral2}, we get, 
for $a$ in $A$, 
$$\xi(a)-\varphi(\gamma(a))+\varphi(\sigma(a))+c\in \mathbb Z,$$
that is, $\xi-\nabla\varphi+c{\bf 1}\in\Lambda$. The conclusion follows.
\end{proof}

The main step in the proof of Theorem \ref{LLT} is the next lemma which essentially deals with the case $p=1$.

\begin{Lem} \label{LLTp=1}
Let $\mathcal G=(Q,A,\sigma,\gamma)$ be a connected directed graph. 
Fix a compact subset $K$ of 
$\{y\in (\nabla \mathcal V)^\perp|\psi_{\mathcal G}(y)>0\}$.
Then, there exists a sequence $(\varepsilon_n)_{n\geq 0}$ of positive real numbers such that $\varepsilon_n\td{n}{\infty}0$ with the following property. For any $q$ in $Q$, $q'$ in $Q_0^q$, $n\geq 0$, $x$ in $\Lambda\cap (\nabla \mathcal V)^\perp$ with $\langle {\bf 1},x\rangle=pn$ and 
$(pn)^{-1}x\in K$, we have 
$$\left|(2\pi pn)^{\frac{r}{2}}\varsigma(x)e^{-\psi(x)}N_{pn}(x,q,q')-
pe^{\langle \theta,\mathscr R(q')-\mathscr R(q)\rangle}\varphi_\theta(q')f_\theta(q)
\right|\leq\varepsilon_n,$$
where $\theta$ is any element of $\mathcal E$ such that $\lambda(\theta)=1$ and $\psi(x)=\langle \theta,x\rangle$.
\end{Lem}

\begin{proof} We fiw $x$ as in the statement and we set $y=(pn)^{-1}x$ which is an element of $K$.
Let $\theta$ be an element of $\mathcal E$ with $\lambda(\theta)=1$ and $\psi(x)=\langle\theta,x\rangle$.

For $n\geq 0$, $q$ in $Q$ and $q'$ in $Q^q_0$, consider the finite measure 
$$\nu_{n,x}^{q,q'}=\sum_{w\in W_{pn}^{q,q'}}e^{-\langle\theta,\mathscr P(w)\rangle}\delta_{\mathscr P(w)-\mathscr R(q')+\mathscr R(q)-x}.$$
This measure is supported on $\Lambda\cap \mathcal E_0=\Lambda\cap (\nabla \mathcal V+\mathbb R{\bf 1})^\perp$. 
We will prove the lemma by evaluating the number
\begin{equation}\label{smartmeasure}
\nu_{n,x}^{q,q'}(0)=
e^{-\psi(x)+\langle \theta,\mathscr R(q)-\mathscr R(q')\rangle}N_{pn}(x,q,q').\end{equation}
We will use the Fourier inversion formula.

Indeed, for $\xi$ in $\mathcal E$, define the characteristic function
$$\widehat{\nu}_{n,x}^{q,q'}(\xi)=\int_{\mathcal E}e^{-2i\pi\langle\xi,z\rangle}\de\nu^{q,q'}_{n,x}(z)
=\sum_{z\in \Lambda\cap (\nabla \mathcal V\oplus{\bf 1})^\perp}e^{-2i\pi\langle\xi,z\rangle}\nu^{q,q'}_{n,x}(z).$$
Note that this function is invariant under translations by $\Lambda+\nabla \mathcal V+\mathbb {\bf 1}$.
Then, the Fourier inversion formula (which in this discrete case is elementary) says that 
\begin{equation}\label{Fourierinv1}
\nu_{n,x}^{q,q'}(0)=\int_{\mathcal E/(\Lambda+\nabla \mathcal V+\mathbb {\bf 1})}\widehat{\nu}_{n,x}^{q,q'}(\xi)\de\xi,\end{equation}
where we have equipped the compact abelian group $\mathcal E/(\Lambda+\nabla \mathcal V+\mathbb {\bf 1})$ 
with its unique translation invariant Borel probability measure.

To conclude, we will use the language of transfer operators to give an asymptotic expansion of $\widehat{\nu}_{n,x}^{q,q'}(\xi)$. 
Indeed, for $\xi$ in $\mathcal E$, we have
$$\widehat{\nu}_{n,x}^{q,q'}(\xi)=e^{2i\pi\langle \xi,x-\mathscr R(q)+\mathscr R(q')\rangle}L_{\theta+2i\pi\xi}^{pn}({\bf 1}_{q'})(q).$$

First, let us use this formula for replacing the integral in \eqref{Fourierinv1} by an integral over a small neighbourhood of $0$.
By Lemma \ref{cplxsptctral}, for any $\xi$ in $\mathcal E\smallsetminus (\Lambda+\nabla \mathcal V+\mathbb {\bf 1})$, the operator 
$L_{\theta+2i\pi\xi}$ has spectral radius $<1$. Therefore, since by Lemma \ref{root}, the spectral radius is a continuous function, for any neighbourhood $U$ of $0$ in $\mathcal E/(\Lambda+\nabla \mathcal V+\mathbb {\bf 1})$, we may find $0\leq \alpha_U<1$ and $C_U>0$ such that, unformly for $y$ in $K$, we have, from \eqref{Fourierinv1},
\begin{equation}\label{Fourierinv2}
\left|\nu_{n,x}^{q,q'}(0)-\int_{U}\widehat{\nu}_{n,x}^{q,q'}(\xi)\de\xi\right|\leq C_U\alpha_U^n.\end{equation}

Now, by Corollary \ref{asymptotic1}, we can find $C,\varepsilon>0$ such that,
$$|L_\theta^{pn}({\bf 1}_{q'})(q)-\lambda(\theta)^{pn}p\varphi_\theta(q')f_\theta(q)|\leq C\varepsilon^n.$$ 
Again, $\varepsilon$ can be chosen to be uniform when $y$ varies in $K$.
Besides, thanks to Lemma \ref{root} and Lemma \ref{residue}, the objects in the inequality above are analytic functions of $\theta$ and we may find a neighbourhood $W$ of $\theta$ in $\mathcal E$ such that, up to lowering $\varepsilon$,
for $\xi$ in $W$, we have
$$|\widehat{\nu}_{n,x}^{q,q'}(\xi)-\lambda(\theta+2i\pi\xi)^{pn}
e^{2i\pi\langle \xi,x-\mathscr R(q)+\mathscr R(q')\rangle}p\varphi_{\theta+2i\pi\xi}(q')f_{\theta+2i\pi\xi}(q)|\leq 
C\varepsilon^n.$$
Up to shrinking $U$, we may assume that $U$ is of the form $U'+\Lambda+\nabla \mathcal V+\mathbb {\bf 1}$, where $U'$ is a neighbourhood of $\{0\}$ in $(\nabla \mathcal V+\mathbb {\bf 1})^\perp$ such that $U'\subset W$. Thus, from \eqref{Fourierinv2},
for some $C'>0$ and $0<\beta<1$, we get 
\begin{multline}\label{Fourierinv3}
\left|\nu_{n,x}^{q,q'}(0)-\int_{U'}\lambda(\theta+2i\pi\xi)^{pn}
e^{2i\pi\langle \xi,x-\mathscr R(q)+\mathscr R(q')\rangle}p\varphi_{\theta+2i\pi\xi}(q')f_{\theta+2i\pi\xi}(q)\de\xi\right|\\
\leq C'\beta^n,\end{multline}
where we have equipped $(\nabla \mathcal V+\mathbb {\bf 1})^\perp$ with the Lebesgue measure such that $\Lambda\cap (\nabla \mathcal V+\mathbb {\bf 1})^\perp$ has covolume $1$.

In this formula, we can eliminate the dependence in $\xi$ of certain terms. Indeed, by using the computation of the derivatives of $\lambda$ from Proposition \ref{lambdaderivative}, if we take $U'$ small enough, for $\xi$ in $U'$, we get 
\begin{equation}\label{pdctbdd}\left|\lambda(\theta+2i\pi\xi)
e^{2i\pi\langle \xi,y\rangle}\right|\leq \exp(-\pi^2\de^2_\theta\log\lambda(\xi,\xi)),\end{equation}
hence, by using the change of variable $\eta=\sqrt{n}\xi$ in the integral, we obtain
$$\int_{U'}\|\xi\|\left|\lambda(\theta+2i\pi\xi)^{n}
e^{2i\pi n\langle \xi,y\rangle}\right|\de\xi\ll n^{-\frac{r+1}{2}}.$$
Thus, from \eqref{Fourierinv3}, using that all objects below the integral are analytic functions,
\begin{equation}\label{Fourierinv4}
\left|\nu_{n,x}^{q,q'}(0)-p\varphi_{\theta}(q')f_\theta(q)\int_{U'}\lambda(\theta+2i\pi\xi)^{pn}
e^{2i\pi pn\langle \xi,y\rangle}\de\xi \right| \ll n^{-\frac{r+1}{2}},\end{equation}
uniformly for $y$ in $K$.

We will conclude by replacing the integrand of the above by a Gaussian law. First, we lower the size of the domain of integration. Keeping in mind that \eqref{pdctbdd} holds and using again the change of variable $\eta=\sqrt{n}\xi$, we obtain
\begin{multline*}\int_{\substack{\xi\in U'\\ \|\xi\|\geq \left(\frac{\log n}{n}\right)^{\frac{1}{2}}}}\left|\lambda(\theta+2i\pi\xi)^{n}
e^{2i\pi n\langle \xi,y\rangle}\right|\de\xi\\
\leq n^{-\frac{r}{2}}\int_{\|\eta\|\geq (\log n)^{\frac{1}{2}}} \exp(-\pi^2 \de^2_\theta\log\lambda(\eta,\eta))\de\eta
\ll (\log n)^r n^{-\frac{r+1}{2}}.\end{multline*}
Now, we are reduced to studying the quantity
$$\int_{\|\xi\|\leq \left(\frac{\log n}{n}\right)^{\frac{1}{2}}}\lambda(\theta+2i\pi\xi)^{n}
e^{2i\pi n\langle \xi,y\rangle}\de\xi.$$
For $\xi$ as under the integral, by Proposition \ref{lambdaderivative}, we have 
$$|\log(\lambda(\theta+2i\pi\xi)
e^{2i\pi\langle \xi,y\rangle})+2\pi^2\de^2_\theta\log\lambda(\xi,\xi)|\ll \|\xi\|^3\leq \left(\frac{\log n}{n}\right)^{\frac{3}{2}}$$
(where we have used the principal determination of the logarithm).
This gives
$$|n\log(\lambda(\theta+2i\pi\xi)
e^{2i\pi\langle \xi,y\rangle})+2n\pi^2\de^2_\theta\log\lambda(\xi,\xi)|\ll (\log n)^{\frac{3}{2}}n^{-\frac{1}{2}},$$
hence, as the exponential function is Lipshitz continuous in the neighbourhood of $0$,
$$\int_{\|\xi\|\leq \left(\frac{\log n}{n}\right)^{\frac{1}{2}}}|\lambda(\theta+2i\pi\xi)^{n}
e^{2i\pi n\langle \xi,y\rangle}-e^{-2n\pi^2\de^2_\theta\log\lambda(\xi,\xi)}|\de\xi
\ll (\log n)^{\frac{r+3}{2}}n^{-\frac{r+1}{2}}.$$
By standard integral computations, we have 
$$\int_{\mathcal E_0} e^{-2\pi^2\de^2_\theta\log\lambda(\eta,\eta)}\de\eta=\frac{1}{(2\pi)^{\frac{r}{2}}\varsigma(x)}.$$
Thus, using again the change of variable $\eta=\sqrt{n}\xi$ gives 
$$\left|\int_{\|\xi\|\leq \left(\frac{\log n}{n}\right)^{\frac{1}{2}}}
e^{-2n\pi^2\de^2_\theta\log\lambda(\xi,\xi)}
\de\xi-\frac{1}{(2\pi n)^{\frac{r}{2}}\varsigma(x)}\right| \ll (\log n)^{r}n^{-\frac{r+1}{2}}.$$
Therefore, from \eqref{Fourierinv4}, we obtain 
$$
\left|\nu_{n,x}^{q,q'}(0)-
\frac{p\varphi_{\theta}(q')f_\theta(q)}{(2\pi pn)^{\frac{r}{2}}\varsigma(x)}\right|
\ll (\log n)^{\max(r,\frac{r+3}{2})}n^{-\frac{r+1}{2}}.$$
In view of \eqref{smartmeasure}, the lemma follows.
\end{proof}

The general case of Theorem \ref{LLT} will now follow from the latter case and an elementary induction argument.

\begin{proof}[Proof of Theorem \ref{LLT}]
We will prove the result for $n$ in $j+p\mathbb Z$ by induction on $0\leq j\leq p-1$.

For $j=0$, this is Lemma \ref{LLTp=1}.

Assume $j\geq 1$ and the result is true for $j-1$. 
For $n\geq 1$, $q,q'$ in $Q$ and $x$ in $\mathcal E$, the definition in \eqref{localcard} gives
$$N_{n}(x,q,q')=\sum_{\substack{a\in A\\ \sigma(a)=q}}N_{n-1}(x-{\bf 1}_a+\mathscr R(\gamma(a))-\mathscr R(q),\gamma(a),q').$$
Assume that $n$ is in $j+p\mathbb Z$ and that $q'$ is in $Q_j^{q}$. Then, for $a$ as above, we have
$q'\in A_{j-1}^{\gamma(a)}$. Besides, if $\langle {\bf 1},x-\mathscr R(q)+\mathscr R(q')\rangle=n$, we have
\begin{multline*}\langle {\bf 1},x-{\bf 1}_a+\mathscr R(\gamma(a))-\mathscr R(q)-\mathscr R(\gamma(a))+\mathscr R(q')\rangle=\langle {\bf 1},x-{\bf 1}_a-\mathscr R(q)+\mathscr R(q')\rangle\\
=n-1.\end{multline*}
Also, the construction in Lemma \ref{straight} warrants that, if $x$ belongs to $(\nabla \mathcal V)^\perp\cap \Lambda$, so does 
$x-{\bf 1}_a+\mathscr R(\gamma(a))-\mathscr R(q)$; if $n^{-1} x$ lives in a compact subset of $\{y\in (\nabla \mathcal V)^\perp|\psi_{\mathcal G}(y)>0\}$, so does 
$(n-1)^{-1}(x-{\bf 1}_a+\mathscr R(\gamma(a))-\mathscr R(q))$. Therefore, the induction assumption together with the continuity of the function $\sigma$ tells us that $N_n(x,q,q')$ is uniformly equivalent to 
$$\frac{p\varphi_\theta(q')e^{\langle \theta,\mathscr R(q')\rangle}}{(2\pi n)^{\frac{r}{2}}\varsigma(x)}
\sum_{\substack{a\in A\\ \sigma(a)=q}}e^{\psi(x-{\bf 1}_a+\mathscr R(\gamma(a))-\mathscr R(q))-\langle \theta,\mathscr R(\gamma(a))\rangle} f_\theta(\gamma(a)).$$
Since $\theta$ is the derivative of $\psi$ at $x$, in view of the assumption, for $a$ as above, we have
$$|\psi(x-{\bf 1}_a+\mathscr R(\gamma(a))-\mathscr R(q))-\psi(x)-\langle \theta,-{\bf 1}_a+\mathscr R(\gamma(a))-\mathscr R(q)\rangle|\ll n^{-1},$$
so that $N_n(x,q,q')$ is uniformly equivalent to 
$$e^{\psi(x)}\frac{p\varphi_\theta(q')e^{\langle \theta,\mathscr R(q')-\mathscr R(q)\rangle}}{(2\pi n)^{\frac{r}{2}}\varsigma(x)}
\sum_{\substack{a\in A\\ \sigma(a)=q}}e^{-\theta(a)} f_\theta(\gamma(a)).$$
The conclusion follows as by definition, $\sum_{\substack{a\in A\\ \sigma(a)=q}}e^{-\theta(a)} f_\theta(\gamma(a))=f_\theta(q)$.
\end{proof}

\section{The use of sofic languages}
\label{secLLT2}

We will now introduce languages defined by automata which are an extension of the notion of languages defined by directed graphs. The counting result established in Theorem \ref{LLT} can be extended to this framework.

Assume now we are given a directed graph $\mathcal G=(Q,A,\sigma,\gamma)$. We define a labelled graph as the additional data of a finite set $B$ (which we view as a new alphabet) together with a surjective map $\pi:A\rightarrow B$. We will think to $B$ as being a set of labels on the edges of the directed graph $\mathcal G$. 
We will always assume the labelling to be deterministic, meaning that, for every $q$ in $Q$, the map $\pi$ is injective on the set of arrows starting from $q$, that is, the set $\{a\in A|\sigma(a)=q\}$. In other words, different arrows emanating from the same vertex have different labels. Such a graph is called right resolving in \cite[Def. 3.3.1]{LM} and it is known that every sofic subshift has a right resolving presentation \cite[Th. 3.3.2]{LM}.

\begin{Ex} \label{fibonacci6}
The Fibonacci labelled graph is obtained by labelling the Fibonacci graph of Figure \ref{fibonaccifig} with two colors $b_1$ and $b_2$ as in Figure \ref{fibonaccifig2}.
\begin{figure}\begin{center}
\begingroup%
  \makeatletter%
  \providecommand\color[2][]{%
    \errmessage{(Inkscape) Color is used for the text in Inkscape, but the package 'color.sty' is not loaded}%
    \renewcommand\color[2][]{}%
  }%
  \providecommand\transparent[1]{%
    \errmessage{(Inkscape) Transparency is used (non-zero) for the text in Inkscape, but the package 'transparent.sty' is not loaded}%
    \renewcommand\transparent[1]{}%
  }%
  \providecommand\rotatebox[2]{#2}%
  \newcommand*\fsize{\dimexpr\f@size pt\relax}%
  \newcommand*\lineheight[1]{\fontsize{\fsize}{#1\fsize}\selectfont}%
  \ifx\svgwidth\undefined%
    \setlength{\unitlength}{197.25bp}%
    \ifx\svgscale\undefined%
      \relax%
    \else%
      \setlength{\unitlength}{\unitlength * \real{\svgscale}}%
    \fi%
  \else%
    \setlength{\unitlength}{\svgwidth}%
  \fi%
  \global\let\svgwidth\undefined%
  \global\let\svgscale\undefined%
  \makeatother%
  \begin{picture}(1,0.47528517)%
    \lineheight{1}%
    \setlength\tabcolsep{0pt}%
    \put(0,0){\includegraphics[width=\unitlength,page=1]{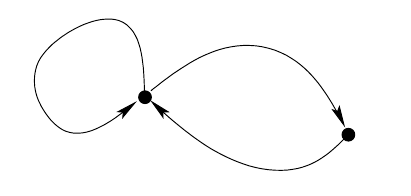}}%
    \put(0.32317608,0.16412644){\color[rgb]{0,0,0}\makebox(0,0)[lt]{\lineheight{1.25}\smash{\begin{tabular}[t]{l}$q_1$\end{tabular}}}}%
    \put(0.87780719,0.13158901){\color[rgb]{0,0,0}\makebox(0,0)[lt]{\lineheight{1.25}\smash{\begin{tabular}[t]{l}$q_2$\end{tabular}}}}%
    \put(0.06330141,0.38266178){\color[rgb]{0,0,0}\makebox(0,0)[lt]{\lineheight{1.25}\smash{\begin{tabular}[t]{l}$b_1$\end{tabular}}}}%
    \put(0.55777007,0.38450641){\color[rgb]{0,0,0}\makebox(0,0)[lt]{\lineheight{1.25}\smash{\begin{tabular}[t]{l}$b_2$\end{tabular}}}}%
    \put(0.59768416,0.02313303){\color[rgb]{0,0,0}\makebox(0,0)[lt]{\lineheight{1.25}\smash{\begin{tabular}[t]{l}$b_1$\end{tabular}}}}%
  \end{picture}%
\endgroup%

\caption{The Fibonacci labelled graph}
\label{fibonaccifig2}
\end{center}\end{figure}
\end{Ex} 

Given such a labelling, for $n\geq 0$, we still write $\pi$ for the associated surjective map $A^n\rightarrow B^n$. We denote by $X$ the language obtained from $W$ by replacing letters in $A$ with their images by $\pi$, that is $X=\bigcup_{n\geq 0} X_n$ where, for $n\geq 0$, $X_n=\pi(W_n)$. In the same way, for $q,q'$ in $Q$ and $n\geq 0$, we write $X_n^{q,q'}=\pi(W_n^{q,q'})$ for the set of words in $X$ which may be obtained through path starting from $q$ and terminating at $q'$. Note that the assumption that the labelling is deterministic directly gives

\begin{Lem}\label{piinjective} Let $\mathcal A=(Q,A,B,\sigma,\gamma,\pi)$ be a deterministic labelled graph. Then, for any $q,q'$ in $Q$ and $n\geq 0$ the map $\pi$ induces a bijection $W_n^{q,q'}\rightarrow X_n^{q,q'}$.
\end{Lem}

Let $F$ be the space of all real valued functions on the set $B$. 
The language $X$ comes with a natural map $\mathscr Q:X\rightarrow F$ which counts the number of occurences of every letter in a given word. Still denote by $\pi:\mathcal E\rightarrow F$ the natural map defined by, for $f$ in $\mathcal E$ and $b$ in $B$, 
$$\pi f(b)=\sum_{\substack{a\in A\\ \pi(a)= b}}f(a).$$
Then, by construction, for $w$ in $W$, we have $\pi \mathscr P(w)=\mathscr Q(\pi w)$. 

Denote by $\psi_\mathcal A$ the indicator of growth of the image under $\mathscr Q$ of the counting measure on set $\mathscr Q(X)=\pi \mathscr P(w)$ in $F$. 
We obtain

\begin{Lem} \label{psiautomaton}
Let $\mathcal A=(Q,A,B,\sigma,\gamma,\pi)$ be a deterministic labelled graph such that the directed graph $\mathcal G=(Q,A\sigma,\gamma)$ is connected. Then the function $\mathcal \psi_{\mathcal A}$ is concave and, for every $y$ in $F$, we have 
$$\psi_{\mathcal A}(y)=\sup_{\substack{x\in \mathcal E\\ \pi(x)= y}}\psi_{\mathcal G}(x).$$
If $\psi_{\mathcal A}(y)>0$, this supremum is attained exactly once.
\end{Lem}

\begin{proof} Fix $q,q'$ in $Q$. As the graph $\mathcal G$ is connected, the function $\psi_{\mathcal G}$ is the indicator of growth of the measure $\mu$ which is the image under $\mathscr P$ of the counting measure on the set $W^{q,q'}$. In the same way, the function $\psi_{\mathcal A}$ is the indicator of growth of the measure $\nu$ which is the image under $\mathscr P$ of the counting measure on the set $X^{q,q'}$. In view of Lemma \ref{piinjective} above, we have $\pi_*\mu=\nu$. Thus, we will aim at applying Corollary \ref{projconcave}.

Indeed, thanks to Lemma \ref{psicounting}, we know that the cone $\{x\in \mathcal E|\psi_{\mathcal G}(x)>-\infty\}$ is the asymptotic cone to $\mathscr P(W)$. By construction, $\mathscr P(W)$ is contained in the set of all nonnegative functions on $A$, hence $\psi_{\mathcal G}$ is $-\infty$ on all functions which take at least one negative value. In particular, if $x\neq 0$ is in $\ker\pi$, we must have $\psi_{\mathcal G}(x)=-\infty$. The formula for $\psi_{\mathcal A}$ now directly follows from Corollary \ref{projconcave} and this formula implies that $\psi_{\mathcal A}$ is concave. The uniqueness statement follows from the strict concavity properties of $\psi_{\mathcal G}$ from Corollary \ref{psiregular}. 
\end{proof}

\begin{Ex}\label{fibonacci7}
For the Fibonacci labelled graph of Example \ref{fibonacci6}, the map $\pi$ may be written as 
$$\pi:\mathbb R^3\rightarrow\mathbb R^2,(x_1,x_2,x_2)\mapsto y=(x_1+x_3,x_2),$$
hence, by using the formula from Example \ref{fibonacci1},
\begin{align*}\psi_\mathcal A(y)&=(y_1-y_2)\log\frac{y_1}{y_1-y_2}+y_2\log\frac{y_1}{y_2}&&y_1\geq y_2\geq 0\\
&=-\infty&&\mbox{else}.\end{align*}
\end{Ex}

The analysis of Section \ref{secLLT} can be extended to the study of the language $X$. Keeping the notation of this Section, we set, as in \eqref{localcard}, for $y$ in $F$, $n\geq 0$ and $q,q'$ in $Q$,
\begin{multline*}\overline{N}_n(y,q,q')=|\{w\in W^{q,q'}_{n}|\pi (\mathscr P(w)-\mathscr R(q')+\mathscr R(q))=y\}|
\\=|\{w\in W^{q,q'}_{n}|\mathscr Q(\pi w)-\pi \mathscr R(q')+\pi \mathscr R(q)=y\}|.\end{multline*}

To state an analogue of Theorem \ref{LLT}, we also need to introduce a new version of the function $\varsigma$. 
Denote by $\pi^*:F\rightarrow \mathcal E$ the adjoint map of $\pi$, that is, for $f$ in $F$ and $a$ in $A$,
$$\pi^*f(a)=f(\pi a).$$

Let $y$ be in $F$ with $\psi_{\mathcal A}(y)>0$. By Lemma \ref{psiautomaton}, there exists a unique $x$ in $\mathcal E$ with $\pi(x)=y$ and $\psi_\mathcal G(x)=\psi_\mathcal A(y)$.
By Corollary \ref{psiregular}, there exists $\theta$ in $\mathcal E$ 
with $\lambda(\theta)=0$ and $\psi_{\mathcal G}(x)=\langle\theta,x\rangle$ and $\theta$ is unique up to translation by an element of $\nabla \mathcal V$. Note that the construction of $x$ implies that $\theta$ actually belongs to $\pi^*F$. 

As in Section \ref{secLLT}, consider the quadratic form $\chi_\theta$ on the space $\mathcal E_0=(\mathbb R{\bf 1}\oplus \nabla \mathcal V)^\perp$ which is defined in Theorem \ref{CLT}. Now, the group $\Lambda\cap \mathcal E_0\cap \pi^*F$ is a lattice in $\mathcal E_0\cap \pi^* F$. We set $\overline{\varsigma}(y)$ to be $|\det (\chi_\theta)_{|\mathcal E_0\cap \pi^*F}|^{\frac{1}{2}}$ where the determinant is evaluated with respect to a basis of the lattice 
$\Lambda\cap \mathcal E_0\cap \pi^* F$.
We let $s$ be the dimension of the latter vector space $\mathcal E_0\cap \pi^* F$.

\begin{Thm} \label{LLT2}
Let $\mathcal A=(Q,A,B,\sigma,\gamma,\pi)$ be a deterministic labelled graph such that 
$\mathcal G=(Q,A,\sigma,\gamma)$ is a connected directed graph. 
Fix a compact subset $K$ of 
$\{z\in F|\psi_{\mathcal G}(z)>0\}$.
Then, there exists a sequence $(\varepsilon_n)_{n\geq 0}$ of positive real numbers such that $\varepsilon_n\td{n}{\infty}0$ with the following property. For any $n\geq 0$, $q,q'$ in $Q$ with $q'\in Q_n^q$ and $y$ in $\pi(\Lambda\cap (\nabla \mathcal V)^\perp)$ 
with $\langle {\bf 1},y-\pi \mathscr R(q)+\pi \mathscr R(q')\rangle=n$ and 
$n^{-1}y\in K$, we have 
$$\left|(2\pi n)^{\frac{s}{2}}\overline{\varsigma}(y)e^{-\psi_{\mathcal A}(y)}\overline{N}_{n}(y,q,q')-
pe^{\langle \theta,\mathscr R(q')-\mathscr R(q)\rangle}\varphi_\theta(q')f_\theta(q)
\right|\leq\varepsilon_n,$$
where $\theta$ is defined as above.
\end{Thm}

\end{document}